\documentclass[11pt,english]{article}
\usepackage[T1]{fontenc}
\usepackage[latin9]{inputenc}
\usepackage[letterpaper]{geometry}
\usepackage{amsmath}
\usepackage{graphicx}
\usepackage{amssymb}

\makeatletter

\providecommand{\tabularnewline}{\\}

\usepackage{array}

\makeatletter

\providecommand{\tabularnewline}{\\}

\newtheorem{theorem}{Theorem}
\newtheorem{lemma}{Lemma}

\newtheorem{definition}{Definition}
\newtheorem{corollary}{Corollary}
\newtheorem{proposition}{Proposition}
\newtheorem{remark}{Remark}

\newenvironment{proof}{\begin{trivlist} \item[]{\bf Proof. }}{\hspace*{\fill}$\rule{.3\baselineskip}{.35\baselineskip}$\end{trivlist}}
\newenvironment{proof1}{
    \noindent {\bf Proof }}{\hfill$\Box$}

\newcommand{\R}{\mathbb{R}}
\newcommand{\C}{\mathbb{C}}
\newcommand{\Z}{\mathbb{Z}}
\newcommand{\N}{\mathbb{N}}

\makeatother

\usepackage{babel}

\date{}

\makeatother

\usepackage{babel}

\begin{document}

\title{\textbf{Internal modes of discrete solitons near the anti-continuum
limit of the dNLS equation}}

\author{Dmitry Pelinovsky and Anton Sakovich\\
 {\small Department of Mathematics, McMaster University, Hamilton
ON, Canada, }}
\maketitle
\begin{abstract}
Discrete solitons of the discrete nonlinear Schrödinger (dNLS) equation
are compactly supported in the anti-continuum limit of the zero coupling
between lattice sites. Eigenvalues of the linearization of the dNLS
equation at the discrete soliton determine its spectral stability.
Small eigenvalues bifurcating from the zero eigenvalue near the anti-continuum
limit were characterized earlier for this model. Here we analyze the
resolvent operator and prove that it is bounded in the neighborhood
of the continuous spectrum if the discrete soliton is simply connected
in the anti-continuum limit. This result rules out existence of internal
modes (neutrally stable eigenvalues of the discrete spectrum) near
the anti-continuum limit. 
\end{abstract}

\section{Introduction}

The {\em discrete nonlinear Schrödinger (dNLS) equation} is a mathematical
model of many physical phenomena including the Bose--Einstein condensation
in optical lattices, propagation of optical pulses in coupled waveguide
arrays, and oscillations of molecules in DNAs \cite{Kevrekidis}.
{\em Discrete solitons} (stationary localized solutions) are used
to interpret the results of physical experiments and to characterize
global dynamics of the dNLS equation with decaying initial data.

Discrete solitons are compactly supported in the {\em anti-continuum
limit} of the zero coupling between lattice sites. Different families
of discrete solitons can be uniquely characterized near the anti-continuum
limit from a number of limiting configurations \cite{Alfimov}. This
is the main reason why the anti-continuum limit has been studied in
many details after the pioneer works of Eilbeck {\em at al.} \cite{Eilbeck}
and Aubry \& Abramovici \cite{Aubry}. The existence of discrete solitons
(also called {\em discrete breathers} in the context of the discrete
Klein--Gordon equation) was rigorously justified with {\em implicit
function theorem arguments} by MacKay \& Aubry \cite{MA94}. Their
work on existence of discrete solitons led to further progress in
understanding their stability properties as well as nonlinear dynamics
of nonlinear lattices \cite{Aubry1,Aubry2,Bambusi}.

{\em Spectral stability} of discrete solitons is determined by
eigenvalues of the discrete spectrum of an associated linearized operator
because its continuous spectrum is neutrally stable. Unstable eigenvalues
can be fully characterized near the anti-continuum limit because they
bifurcate from the zero eigenvalue of finite multiplicity and the
zero eigenvalue is isolated from the continuous spectrum. Characterization
of unstable eigenvalues for each family of discrete solitons bifurcating
from a compact limiting solution was obtained by Pelinovsky {\em
et al.} \cite{PKF1} with an application of {\em Lyapunov--Schmidt
reduction technique}. Beside the unstable eigenvalues, the same technique
was used to characterize a number of neutrally stable eigenvalues
of negative energy (also called eigenvalues of {\em negative Krein
signature}) which bifurcate from the same zero eigenvalue. These
isolated eigenvalues of negative energy may become unstable far from
the anti-continuum limit because of collisions with eigenvalues of
positive energy (also called {\em internal modes}) or with the
continuous spectrum of the linearized operator. Isolated eigenvalues
of negative energy may also induce nonlinear instability if their
multiples belong to the continuous spectrum \cite{Cuccagna}.

In the same anti-continuum limit, another bifurcation occurs beyond
the applicability of the Lyapunov--Schmidt reduction technique: a
pair of {\em semi-simple} nonzero eigenvalues of {\em infinite
multiplicity} transforms into a pair of continuous spectral bands
of small width. This transformation may produce additional eigenvalues
of the discrete spectrum similar to what happens for the discrete
kinks (which are non-compact solutions of the nonlinear lattice in
the anti-continuum limit) \cite{PK08}. No complex unstable eigenvalues
may bifurcate from the semi-simple nonzero eigenvalues of infinite
multiplicity because such eigenvalues are excluded by the count of
unstable eigenvalues in \cite{PKF1}. Nevertheless, internal modes
may in general be expected outside the continuous spectrum.

It is important to know the details on existence of internal modes
because of several reasons. First, these internal modes may collide
with eigenvalues of negative energy to produce the Hamilton-Hopf instability
bifurcations \cite{PKF1}. Second, analysis of asymptotic stability
of discrete solitons depends on the number and location of the internal
modes \cite{CT,KPS}. Third, the presence of internal modes may result
in long-term quasi-periodic oscillations of discrete solitons \cite{Cuccagna-breather}.

In this paper, we address bifurcations of internal modes from semi-simple
nonzero eigenvalues of infinite multiplicity. We continue the resolvent
operator across the continuous spectrum and prove that it is bounded
near the end points of the continuous spectrum if the discrete soliton
is {\em simply connected} in the anti-continuum limit, see Definition
\ref{definition-simply-connected}. As a result, no internal modes
exist in the neighborhood of the continuous spectrum. These results
hold for {\em any} discrete soliton of the dNLS equation with {\em
any} power nonlinearity near the anti-continuum limit.

There are multiple numerical evidences that no internal modes exist
near the anti-continuum limit for the {\em fundamental} discrete
soliton, which is supported at a single lattice site in the zero coupling
limit. In particular, this fact is suggested by Figure 1 in Johansson
\& Aubry \cite{JA00} and by Figure 2.5 in Kevrekidis \cite{Kevrekidis}.
Our article presents the first analytical proof of this phenomenon.

The paper is organized as follows. Section 2 reviews results on existence
and stability of discrete solitons near the anti-continuum limit.
Section 3 is devoted to analysis of the resolvent operator with the
limiting compact potentials. Section 4 develops perturbative arguments
for the full resolvent operator. Section 5 considers a case study
for the resolvent operator associated with a non-simply-connected
$2$-site discrete soliton. Appendix A is devoted to the cubic dNLS
equation, for which perturbation arguments are more delicate.

\textbf{Notations.} We denote the bi-infinite sequence $\{u_{n}\}_{n\in\Z}$
by ${\bf u}$. The $l^{p}$ space for sequences is denoted by $l^{p}(\Z)$
and is equipped with the norm \[
\|{\bf u}\|_{l^{p}}:=\left(\sum_{n\in\Z}|u_{n}|^{p}\right)^{1/p},\quad p\geq1.\]
 The algebraically weighted space $l_{s}^{p}(\Z)$ with $s\in\R$
is the $l^{p}(\Z)$ space for the sequence \[
\{(1+n^{2})^{s/2}u_{n}\}_{n\in\Z}.\]
 A disk of radius $\delta>0$ centered at the point $\lambda_{0}\in\C$
on the complex plane is denoted by $B_{\delta}(\lambda_{0})\subset\C$.

\section{Review of results on discrete solitons}

Consider the dNLS equation in the form \begin{equation}
i\dot{u}_{n}+\epsilon(u_{n+1}-2u_{n}+u_{n-1})+|u_{n}|^{2p}u_{n}=0,\quad n\in\Z,\label{dNLS}\end{equation}
 where the dot denotes differentiation in $t\in\R$, $\{u_{n}(t)\}_{n\in\Z}:\R^{\Z}\to\C$
is the set of amplitude functions, and parameters $\epsilon\in\R$
and $p\in\N$ define the coupling constant and the power of nonlinearity.
The anti-continuum limit corresponds to $\epsilon=0$, in which case
the dNLS equation (\ref{dNLS}) becomes an infinite system of uncoupled
differential equations.

Discrete solitons are defined in the form $u_{n}(t)=\phi_{n}e^{it}$,
where the frequency is normalized thanks to the scaling symmetry of
the power nonlinearity. By the standard arguments \cite{PP08} based
on the conserved quantity \begin{equation}
\epsilon\neq0:\quad\bar{\phi}_{n}\phi_{n+1}-\phi_{n}\bar{\phi}_{n+1}={\rm const}\;\;{\rm in}\;\; n\in\Z,\end{equation}
 it is known that if $\{\phi_{n}\}_{n\in\Z}$ decays to zero as $|n|\to\infty$,
then $\{\phi_{n}\}_{n\in\Z}$ is real-valued module to multiplication
by $e^{i\theta}$ for any $\theta\in\R$. The real-valued stationary
solutions are found from the second-order difference equation \begin{equation}
(1-\phi_{n}^{2p})\phi_{n}=\epsilon(\phi_{n+1}-2\phi_{n}+\phi_{n-1}),\quad n\in\Z.\label{stat-dNLS}\end{equation}
 The algebraic system is uncoupled if $\epsilon=0$.

Let us consider solutions of the difference equation (\ref{stat-dNLS})
for $\mbox{\boldmath\ensuremath{\phi}}\in l^{2}(\Z)$. If $\epsilon=0$
and $p\in\N$, the limiting configuration of the discrete soliton
is given by the compact solution \begin{equation}
\epsilon=0:\quad\mbox{\boldmath\ensuremath{\phi}}^{(0)}=\sum_{n\in U_{+}}\mbox{\boldmath\ensuremath{\delta}}_{n}-\sum_{n\in U_{-}}\mbox{\boldmath\ensuremath{\delta}}_{n},\label{limit-solution}\end{equation}
 where $U_{\pm}$ are compact subset of $\Z$ such that $U_{+}\cap U_{-}=\varnothing$
and $\mbox{\boldmath\ensuremath{\delta}}_{n}$ is the standard unit
vector in $l^{2}(\Z)$ expressed via the Kronecker symbol by \[
(\mbox{\boldmath\ensuremath{\delta}}_{n})_{m}=\delta_{n,m},\quad m\in\Z.\]
 We will denote the number of sites in $U_{\pm}$ by $|U_{\pm}|$.
The following proposition gives a unique analytic continuation of
the compact limiting solution (\ref{limit-solution}) to a particular
family of discrete solitons (see \cite{MA94,PP08,PKF1} for the proof).

\begin{proposition} Fix $U_{+},U_{-}\subset\Z$ such that $U_{+}\cap U_{-}=\varnothing$
and $|U_{+}|+|U_{-}|<\infty$. There exists $\epsilon_{0}>0$ such
that the stationary dNLS equation (\ref{stat-dNLS}) with $\epsilon\in(-\epsilon_{0},\epsilon_{0})$
admits a unique solution $\mbox{\boldmath\ensuremath{\phi}}\in l^{2}(\Z)$
near $\mbox{\boldmath\ensuremath{\phi}}^{(0)}\in l^{2}(\Z)$. The
map $(-\epsilon_{0},\epsilon_{0})\ni\epsilon\mapsto\mbox{\boldmath\ensuremath{\phi}}\in l^{2}(\Z)$
is analytic and \begin{equation}
\exists C>0:\quad\|\mbox{\boldmath\ensuremath{\phi}}-\mbox{\boldmath\ensuremath{\phi}}^{(0)}\|_{l^{2}}\leq C|\epsilon|.\label{soliton-power-series}\end{equation}
 Moreover, there are $\kappa>0$ and $C>0$ such that for any $\epsilon\in(-\epsilon_{0},\epsilon_{0})$
\begin{equation}
|\phi_{n}|\leq Ce^{-\kappa|n|},\quad n\in\Z.\label{soliton-decay}\end{equation}
 \label{proposition-existence} \end{proposition}

\begin{remark} Thanks to the exponential decay (\ref{soliton-decay}),
the solution $\mbox{\boldmath\ensuremath{\phi}}\in l^{2}(\Z)$ of
Proposition \ref{proposition-existence} belongs to $\mbox{\boldmath\ensuremath{\phi}}\in l_{s}^{2}(\Z)$
for any $s\geq0$. \label{remark-existence} \end{remark}

By Proposition \ref{proposition-existence}, the solution $\mbox{\boldmath\ensuremath{\phi}}$
for a given $\mbox{\boldmath\ensuremath{\phi}}^{(0)}$ can be expanded
in the power series \begin{equation}
\mbox{\boldmath\ensuremath{\phi}}=\mbox{\boldmath\ensuremath{\phi}}^{(0)}+\sum_{k=1}^{\infty}\epsilon^{k}\mbox{\boldmath\ensuremath{\phi}}^{(k)},\quad\epsilon\in(-\epsilon_{0},\epsilon_{0}),\label{power-series}\end{equation}
 where correction terms $\{\mbox{\boldmath\ensuremath{\phi}}^{(k)}\}_{k\in\N}$
are uniquely defined by a recursion formula.

Spectral stability of the discrete solitons is determined from analysis
of the spectral problem \begin{equation}
L_{+}{\bf u}=-\lambda{\bf w},\quad L_{-}{\bf w}=\lambda{\bf u},\label{spectral-stability}\end{equation}
 where $\lambda\in\C$ is the spectral parameter, $({\bf u},{\bf w})\in l^{2}(\Z)\times l^{2}(\Z)$
is an eigenvector, and $L_{\pm}$ are discrete Schrödinger operators
given by \begin{eqnarray}
\left\{ \begin{array}{l}
(L_{+}{\bf u})_{n}=-\epsilon(u_{n+1}-2u_{n}+u_{n-1})+(1-(2p+1)\phi_{n}^{2p})u_{n},\\
(L_{-}{\bf w})_{n}=-\epsilon(w_{n+1}-2w_{n}+w_{n-1})+(1-\phi_{n}^{2p})w_{n},\end{array}\right.\quad n\in\Z.\end{eqnarray}
 We recall basic definitions and results from the stability analysis
of the spectral problem (\ref{spectral-stability}).

\begin{definition} The eigenvalues of the spectral problem (\ref{spectral-stability})
with ${\rm Re}(\lambda)>0$ (resp. ${\rm Re}(\lambda)=0$) are called
unstable (resp. neutrally stable). If $\lambda\in i\R$ is a simple
isolated eigenvalue, then the eigenvalue $\lambda$ is said to have
a positive energy if $\langle L_{+}{\bf u},{\bf u}\rangle_{l^{2}}>0$
and a negative energy if $\langle L_{+}{\bf u},{\bf u}\rangle_{l^{2}}<0$.
\end{definition}

\begin{remark} If $\lambda\in i\R$ is an isolated eigenvalue and
$\langle L_{+}{\bf u},{\bf u}\rangle_{l^{2}}=0$, then $\lambda$
is not a simple eigenvalue. In this case, the concept of eigenvalues
of positive and negative energies is defined by the diagonalization
of the quadratic form $\langle L_{+}{\bf u},{\bf u}\rangle_{l^{2}}$,
where ${\bf u}$ belongs to the subspace of $l^{2}(\Z)$ associated
to the eigenvalue $\lambda$ of the spectral problem (\ref{spectral-stability})
and invariant under the action of the corresponding linearized operator
(see \cite{ChP} for the relevant theory). \end{remark}

In the anti-continuum limit $\epsilon=0$, the spectrum of $L_{+}$
(resp. $L_{-}$) includes a semi-simple eigenvalue $-2p$ (resp. $0$)
of multiplicity $N=|U_{+}|+|U_{-}|<\infty$ and a semi-simple eigenvalue
$1$ of multiplicity $|\Z\backslash\{U_{+}\cup U_{-}\}|=\infty$.
The spectral problem (\ref{spectral-stability}) has a pair of eigenvalues
$\lambda=\pm i$ of infinite multiplicity and the eigenvalue $\lambda=0$
of geometric multiplicity $N$ and algebraic multiplicity $2N$. The
following proposition describes the splitting of the zero eigenvalue
near the anti-continuum limit for $\epsilon>0$ (see \cite{PKF1}
for the proof).

\begin{proposition} \label{proposition-stability} Fix $U_{+},U_{-}\subset\Z$
such that $U_{+}\cap U_{-}=\varnothing$ and $N:=|U_{+}|+|U_{-}|<\infty$.
Fix $\epsilon>0$ sufficiently small and denote the number of sign
differences of $\{\phi_{n}^{(0)}\}_{n\in U_{+}\cup U_{-}}$ by $n_{0}$. 
\begin{itemize}
\item There are exactly $n_{0}$ negative and $N-1-n_{0}$ small positive
eigenvalues of $L_{-}$ counting multiplicities and a simple zero
eigenvalue.
\item There are exactly $n_{0}$ pairs of small eigenvalues $\lambda\in i\R$
and $N-1-n_{0}$ pairs of small eigenvalues $\lambda\in\R$ of the
spectral problem (\ref{spectral-stability}) counting multiplicities
and a double zero eigenvalue. 
\end{itemize}
\end{proposition}

Proposition \ref{proposition-stability} completes the characterization
of unstable eigenvalues and neutrally stable eigenvalues of negative
energy from negative eigenvalues of $L_{+}$ and $L_{-}$. In particular,
we know from \cite{ChP} that if ${\rm Ker}(L_{+})=\{0\}$, ${\rm Ker}(L_{-})={\rm span}\{\mbox{\boldmath\ensuremath{\phi}}\}$,
and $\langle L_{+}^{-1}\mbox{\boldmath\ensuremath{\phi}},\mbox{\boldmath\ensuremath{\phi}}\rangle_{l^{2}}\neq0$,
then \begin{equation}
\left\{ \begin{array}{lcl}
n(L_{+})-p_{0} & = & N_{r}^{-}+N_{i}^{-}+N_{c},\\
n(L_{-}) & = & N_{r}^{+}+N_{i}^{-}+N_{c},\end{array}\right.\label{count-eigenvalues}\end{equation}
 where $n(L_{\pm})$ denotes the number of negative eigenvalues of
$L_{\pm}$, $N_{i}^{-}$ denotes the number of eigenvalues $\lambda\in i\R$
with negative energy, $N_{c}$ denotes the number of eigenvalues with
${\rm Re}(\lambda)>0$ and ${\rm Im}(\lambda)>0$, $N_{r}^{\pm}$
denotes the number of eigenvalues $\lambda\in\R$ with $\langle L_{+}{\bf u},{\bf u}\rangle_{l^{2}}\gtrless0$,
and \[
p_{0}=\left\{ \begin{array}{l}
1\quad\mbox{{\rm if}}\quad\langle L_{+}^{-1}\mbox{\boldmath\ensuremath{\phi}},\mbox{\boldmath\ensuremath{\phi}}\rangle_{l^{2}}<0,\\
0\quad\mbox{{\rm if}}\quad\langle L_{+}^{-1}\mbox{\boldmath\ensuremath{\phi}},\mbox{\boldmath\ensuremath{\phi}}\rangle_{l^{2}}>0.\end{array}\right.\]

To compute $p_{0}$, we extend the family of discrete solitons by
parameter $\omega$ as solutions of \begin{equation}
(\omega-\phi_{n}^{2p})\phi_{n}=\epsilon(\phi_{n+1}-2\phi_{n}+\phi_{n-1}),\quad n\in\Z.\label{eq-help}\end{equation}
 Differentiation of equation (\ref{eq-help}) in $\omega$ at $\omega=1$
gives \[
\langle L_{+}^{-1}\mbox{\boldmath\ensuremath{\phi}},\mbox{\boldmath\ensuremath{\phi}}\rangle_{l^{2}}=-\langle\partial_{\omega}\mbox{\boldmath\ensuremath{\phi}}|_{\omega=1},\mbox{\boldmath\ensuremath{\phi}}\rangle_{l^{2}}=-\frac{1}{2}\frac{d}{d\omega}\|\mbox{\boldmath\ensuremath{\phi}}\|_{l^{2}}^{2}\biggr|_{\omega=1}=-\frac{N}{2p}+{\cal O}(\epsilon),\]
 where in the last equality we used Proposition \ref{proposition-existence}
and the anti-continuum limit \[
\epsilon=0:\quad\|\mbox{\boldmath\ensuremath{\phi}}(\omega)\|_{l^{2}}^{2}=N\omega^{1/p}.\]
 Therefore, $p_{0}=1$ for small $\epsilon>0$.

By Proposition \ref{proposition-stability}, we have $n(L_{-})=n_{0}$
and $N_{i}^{-}\geq n_{0}$. Also, $n(L_{+})=N$. Using the count (\ref{count-eigenvalues}),
we have for small $\epsilon>0$ \begin{equation}
N_{r}^{+}=0,\quad N_{r}^{-}=N-1-n_{0},\quad N_{i}^{-}=n_{0},\quad N_{c}=0.\label{count-eigenvalues-explicit}\end{equation}

Equality (\ref{count-eigenvalues-explicit}) shows that besides the
small and zero eigenvalues described by Proposition \ref{proposition-stability},
the spectral problem (\ref{spectral-stability}) may only have the
continuous spectrum and the eigenvalues on $i\R$ with positive energy.
These eigenvalues of positive energy are called the {\em internal
modes} and existence of such eigenvalues for small $\epsilon>0$
is the main theme of this article.

\section{The resolvent operator for the limiting configuration}

Let us consider the truncated spectral problem (\ref{spectral-stability})
after $\mbox{\boldmath\ensuremath{\phi}}$ is replaced by $\mbox{\boldmath\ensuremath{\phi}}^{(0)}$.
The resolvent operator is defined from the inhomogeneous system \begin{equation}
\left\{ \begin{array}{l}
-\epsilon(u_{n+1}-2u_{n}+u_{n-1})+u_{n}-(2p+1)\sum_{m\in U_{+}\cup U_{-}}\delta_{n,m}u_{m}+\lambda w_{n}=F_{n},\\
-\epsilon(w_{n+1}-2w_{n}+w_{n-1})+w_{n}-\sum_{m\in U_{+}\cup U_{-}}\delta_{n,m}w_{m}-\lambda u_{n}=G_{n},\end{array}\right.\quad n\in\N,\label{equation-resolvent-1}\end{equation}
 where ${\bf F},{\bf G}\in l^{2}(\Z)$ are given. Since we are interested
in the continuous spectrum and eigenvalues on $i\R$, we set $\lambda=-i\Omega$
and use new coordinates \[
\left\{ \begin{array}{l}
a_{n}:=u_{n}+iw_{n},\quad b_{n}:=u_{n}-iw_{n},\\
f_{n}:=F_{n}+iG_{n},\quad g_{n}:=F_{n}-iG_{n},\end{array}\right.\quad n\in\Z.\]
 The inhomogeneous system (\ref{equation-resolvent-1}) transforms
to the equivalent form \begin{equation}
\left\{ \begin{array}{l}
-\epsilon(a_{n+1}-2a_{n}+a_{n-1})+a_{n}-\sum_{m\in U_{+}\cup U_{-}}\delta_{n,m}((1+p)a_{m}+pb_{m})-\Omega a_{n}=f_{n},\\
-\epsilon(b_{n+1}-2b_{n}+b_{n-1})+b_{n}-\sum_{m\in U_{+}\cup U_{-}}\delta_{n,m}(pa_{m}+(1+p)b_{m})+\Omega b_{n}=g_{n},\end{array}\right.\label{equation-resolvent-2}\end{equation}
 which can be rewritten in the operator form \begin{equation}
L\left[\begin{array}{cc}
{\bf a}\\
{\bf b}\end{array}\right]-\Omega\left[\begin{array}{cc}
{\bf a}\\
{\bf b}\end{array}\right]=\left[\begin{array}{cc}
{\bf f}\\
-{\bf g}\end{array}\right],\quad L=\left[\begin{array}{cc}
-\epsilon\Delta+I-(1+p)V & -pV\\
pV & \epsilon\Delta-I+(1+p)V,\end{array}\right],\label{equation-resolvent-3}\end{equation}
 where $\Delta:l^{2}(\Z)\to l^{2}(\Z)$ is the discrete Laplacian
operator \[
(\Delta u)_{n}:=u_{n+1}-2u_{n}+u_{n-1},\quad n\in\Z\]
 and $V:l^{2}(\Z)\to l^{2}(\Z)$ is the associated compact potential
\[
(Vu)_{n}=\sum_{m\in U_{+}\cup U_{-}}\delta_{n,m}u_{m},\quad n\in\Z.\]

Let $R_{0}(\lambda):l^{2}(\Z)\to l^{2}(\Z)$ be a free resolvent of
the discrete Schrödinger operator $-\Delta$ for $\lambda\notin\sigma(-\Delta)\equiv[0,4]$.
The free resolvent was studied recently by Komech, Kopylova, \& Kunze
\cite{KKK}. The free resolvent operator can be expressed in the Green
function form \begin{equation}
\forall{\bf f}\in l^{2}(\Z):\quad(R_{0}(\lambda){\bf f})_{n}=\frac{1}{2i\sin z(\lambda)}\sum_{m\in\Z}e^{-iz(\lambda)|n-m|}f_{m},\label{green-function}\end{equation}
 where $z(\lambda)$ is a unique solution of the transcendental equation
for $\lambda\notin[0,4]$ \begin{equation}
2-2\cos z(\lambda)=\lambda,\quad{\rm Re}z(\lambda)\in[-\pi,\pi),\;\;{\rm Im}z(\lambda)<0.\label{transcendental-equation-1}\end{equation}

The limiting absorption principle (see, e.g., Pelinovsky \& Stefanov
\cite{PS08}) states that a bounded operator $R_{0}(\lambda):l^{2}(\Z)\to l^{2}(\Z)$
for $\lambda\notin[0,4]$ admits the limits \[
R_{0}^{\pm}(\omega)=\lim_{\mu\downarrow0}R_{0}(\omega\pm i\mu):l_{\sigma}^{2}(\Z)\to l_{-\sigma}^{2}(\Z),\quad\sigma>\frac{1}{2}\]
 for any fixed $\omega\in(0,4)$.

The limiting free resolvent operators $R_{0}^{\pm}(\omega)$ can also
be expressed in the Green function form \begin{equation}
\forall{\bf f}\in l^{1}(\Z):\quad(R_{0}^{\pm}(\omega){\bf f})_{n}=\frac{1}{2i\sin\theta_{\pm}(\omega)}\sum_{m\in\Z}e^{-i\theta_{\pm}(\omega)|n-m|}f_{m},\label{green-function-continuous}\end{equation}
 where $\theta_{\pm}(\omega)=\pm\theta(\omega)$ and $\theta(\omega)$
is a unique solution of the transcendental equation for $\omega\in[0,4]$
\begin{equation}
2-2\cos\theta(\omega)=\omega,\quad{\rm Re}\theta(\omega)\in[-\pi,0],\quad{\rm Im}\theta(\omega)=0.\label{transcendental-equation-2}\end{equation}
 The limiting operators $R_{0}^{\pm}(\omega):l^{1}(\Z)\to l^{\infty}(\Z)$
are bounded for any fixed $\omega\in(0,4)$ but diverge as $\omega\downarrow0$
and $\omega\uparrow4$. These divergences follow from the Puiseux
expansion \begin{equation}
\forall{\bf f}\in l_{2}^{1}(\Z):\quad(R_{0}^{\pm}(\omega){\bf f})_{n}=\frac{1}{2i\theta^{\pm}(\omega)}\sum_{m\in\mathbb{Z}}f_{m}-\frac{1}{2}\sum_{m\in\mathbb{Z}}|n-m|f_{m}+(\hat{R}_{0}^{\pm}(\omega){\bf f})_{n},\label{Puiseux-free}\end{equation}
 where \[
\exists C>0:\quad\|\hat{R}_{0}^{\pm}(\omega){\bf f}\|_{l^{\infty}}\leq C|\theta^{\pm}(\omega)|\|{\bf f}\|_{l_{2}^{1}}.\]

Divergences of $R_{0}^{\pm}(\omega)$ at the end points $\omega=0$
and $\omega=4$ indicate {\em resonances}, which may result in
the bifurcation of new eigenvalues from the continuous spectrum on
$[0,4]$ either for $\lambda<0$ or $\lambda>4$, when $-\Delta$
is perturbed by a small potential in $l^{2}(\Z)$.

Let us denote the solution of the inhomogeneous system (\ref{equation-resolvent-3})
by \begin{equation}
\left[\begin{array}{cc}
{\bf a}\\
{\bf b}\end{array}\right]=R_{L}(\Omega)\left[\begin{array}{cc}
{\bf f}\\
-{\bf g}\end{array}\right],\quad R_{L}(\Omega)=\left[\begin{array}{cc}
R_{11}(\Omega) & R_{12}(\Omega)\\
R_{21}(\Omega) & R_{22}(\Omega)\end{array}\right].\label{equation-resolvent-4}\end{equation}

The following theorem represents the main result of this section.
This theorem is valid for the simply connected sets $U_{+}\cup U_{-}$,
which are defined by the following definition.

\begin{definition} \label{definition-simply-connected} We say that
the set $U_{+}\cup U_{-}$ is simply connected if no elements in $\Z\backslash\{U_{+}\cup U_{-}\}$
are located between elements in $U_{+}\cup U_{-}$. \end{definition}

\begin{theorem} Fix $U_{+},U_{-}\subset\Z$ such that $U_{+}\cap U_{-}=\varnothing$,
$N:=|U_{+}|+|U_{-}|<\infty$, and $U_{+}\cup U_{-}$ is simply connected.
There exist small $\epsilon_{0}>0$ and $\delta>0$ such that for
any fixed $\epsilon\in(0,\epsilon_{0})$ the resolvent operator \[
R_{L}(\Omega):l^{2}(\Z)\times l^{2}(\Z)\to l^{2}(\Z)\times l^{2}(\Z)\]
 is bounded for any $\Omega\notin B_{\delta}(0)\cup[1,1+4\epsilon]\cup[-1-4\epsilon,-1]$.
Moreover, $R_{L}(\Omega)$ has exactly $2N$ poles (counting multiplicities)
inside $B_{\delta}(0)$ and admits the limits \[
R_{L}^{\pm}(\Omega):=\lim_{\mu\downarrow0}R_{L}(\Omega\pm i\mu)\]
 such that for any $\Omega\in[1,1+4\epsilon]\cup[-1-4\epsilon,-1]$
and any $\epsilon\in(0,\epsilon_{0})$, there is $C>0$ such that
\[
\|R_{L}^{\pm}(\Omega)\|_{l_{1}^{1}\times l_{1}^{1}\to l^{\infty}\times l^{\infty}}\leq C\epsilon^{-1}.\]
 \label{theorem-resolvent} \end{theorem}

\begin{remark} The other way to formulate the main theorem is to
say that the end points of the continuous spectrum $\sigma_{c}(L)\equiv[1,1+4\epsilon]\cup[-1-4\epsilon,-1]$
are not resonances and no eigenvalues of the linear operator $L$
may exist outside a small disk $B_{\delta}(0)\subset\C$. The $2N$
eigenvalues inside the small disk $B_{\delta}(0)$ are characterized
in Proposition \ref{proposition-stability}. \end{remark}

Solving the linear system (\ref{equation-resolvent-2}) with the Green
function (\ref{green-function}), we obtain the exact solution for
any $n\in\Z$ \begin{eqnarray}
\left\{ \begin{array}{l}
a_{n}=\frac{1}{2i\epsilon\sin z(\lambda_{+})}\left(\sum_{m\in\Z}e^{-iz(\lambda_{+})|n-m|}f_{m}+\sum_{m\in U_{+}\cup U_{-}}e^{-iz(\lambda_{+})|n-m|}((1+p)a_{m}+pb_{m})\right),\\
b_{n}=\frac{1}{2i\epsilon\sin z(\lambda_{-})}\left(\sum_{m\in\Z}e^{-iz(\lambda_{-})|n-m|}g_{m}+\sum_{m\in U_{+}\cup U_{-}}e^{-iz(\lambda_{-})|n-m|}(pa_{m}+(1+p)b_{m})\right),\end{array}\right.\label{resolvent-a-b}\end{eqnarray}
 where the map $\C\ni\lambda\mapsto z\in\C$ is defined by the transcendental
equation (\ref{transcendental-equation-1}) and \[
\lambda_{\pm}=\frac{\pm\Omega-1}{\epsilon}.\]
 The solution is closed if the set $\{(a_{n},b_{n})\}_{n\in U_{+}\cup U_{-}}$
is found from the linear system of finitely many equations for any
$n\in U_{+}\cup U_{-}$ \begin{eqnarray}
\left\{ \begin{array}{l}
2i\epsilon\sin z(\lambda_{+})a_{n}-\sum_{m\in U_{+}\cup U_{-}}e^{-iz(\lambda_{+})|n-m|}((1+p)a_{m}+pb_{m})=\sum_{m\in\Z}e^{-iz(\lambda_{+})|n-m|}f_{m},\\
2i\epsilon\sin z(\lambda_{-})b_{n}-\sum_{m\in U_{+}\cup U_{-}}e^{-iz(\lambda_{-})|n-m|}(pa_{m}+(1+p)b_{m})=\sum_{m\in\Z}e^{-iz(\lambda_{-})|n-m|}g_{m}.\end{array}\right.\label{system-a-b}\end{eqnarray}

Let us order lattice sites $n\in U_{+}\cup U_{-}$ such that the first
site is placed at $n=0$, the second site is placed at $m_{1}$, the
third site is placed at $m_{1}+m_{2}$, and so on, the last site is
placed at $m_{1}+m_{2}+\cdots+m_{N-1}$, where $N=|U_{+}|+|U_{-}|$
and all $m_{j}>0$. If $U_{+}\cup U_{-}$ is a simply-connected set,
then all $m_{j}=1$.

Let $Q(q_{1},q_{2},\cdots,q_{N-1})$ be the matrix in $\C^{N\times N}$
defined by \begin{equation}
Q(q_{1},q_{2},\cdots,q_{N-1}):=\left[\begin{array}{ccccc}
1 & q_{1} & q_{1}q_{2} & \cdots & q_{1}q_{2}\cdots q_{N-1}\\
q_{1} & 1 & q_{2} & \cdots & q_{2}q_{3}\cdots q_{N-1}\\
q_{1}q_{2} & q_{2} & 1 & \cdots & q_{3}\cdots q_{N-1}\\
\vdots & \vdots & \vdots & \vdots & \vdots\\
q_{1}q_{2}\cdots q_{N-1} & q_{2}\cdots q_{N-1} & q_{3}\cdots q_{N-1} & \cdots & 1\end{array}\right].\label{matrix-M-coefficient}\end{equation}
 Let $q_{j}^{\pm}=e^{-im_{j}z(\lambda_{\pm})}$ and $Q^{\pm}(\Omega,\epsilon):=Q(q_{1}^{\pm},q_{2}^{\pm},\cdots,q_{N-1}^{\pm})$.
The coefficient matrix of the linear system (\ref{system-a-b}) is
given by \begin{equation}
A(\Omega,\epsilon):=\left[\begin{array}{cc}
2i\epsilon\sin z(\lambda_{+})I-(1+p)Q^{+}(\Omega,\epsilon) & -pQ^{+}(\Omega,\epsilon)\\
-pQ^{-}(\Omega,\epsilon) & 2i\epsilon\sin z(\lambda_{-})I-(1+p)Q^{-}(\Omega,\epsilon)\end{array}\right],\label{coefficient-matrix-1}\end{equation}
 where $I$ is an identity matrix in $\C^{N\times N}$.

We split the proof of Theorem \ref{theorem-resolvent} into three
subsections, where solutions of system (\ref{resolvent-a-b}) and
(\ref{system-a-b}) are studied for different values of $\Omega$.

\subsection{Resolvent outside the continuous spectrum}

We consider the resolvent operator $R_{L}(\Omega)$ for a fixed small
$\epsilon\in(0,\epsilon_{0})$. The following lemma shows that $R_{L}(\Omega)$
is a bounded operator from $l^{2}(\Z)\times l^{2}(\Z)$ to $l^{2}(\Z)\times l^{2}(\Z)$
for all $\Omega\in\C$ except three disks of small radii centered
at $\{0,1,-1\}$.

\begin{lemma} There are $\epsilon_{0}>0$ and $\delta,\delta_{\pm}>0$
such that for any $\epsilon\in(0,\epsilon_{0})$, the resolvent operator
$R_{L}(\Omega):l^{2}(\Z)\times l^{2}(\Z)\to l^{2}(\Z)\times l^{2}(\Z)$
is bounded for all $\Omega\in\C\backslash\{B_{\delta}(0)\cup B_{\delta_{+}}(1)\cup B_{\delta_{-}}(-1)\}$.
Moreover, $R_{L}(\Omega)$ has exactly $2N$ poles (counting multiplicities)
inside $B_{\delta}(0)$. \label{lemma-first-part} \end{lemma}

\begin{proof} From the property of the free resolvent operator $R_{0}(\lambda)$,
we know that the Green function in the representation (\ref{resolvent-a-b})
is bounded and exponentially decaying as $|n|\to\infty$ for any $\Omega$
such that $\lambda_{\pm}\notin[0,4]$. This gives $\Omega\notin\sigma_{c}(L)\equiv[1,1+4\epsilon]\cup[-1-4\epsilon,-1]$.
Therefore, $R_{L}(\Omega)$ is bounded map from $l^{2}(\Z)\times l^{2}(\Z)$
to $l^{2}(\Z)\times l^{2}(\Z)$ for any $\Omega\notin\sigma_{c}(L)$
if and only if the system of linear equations (\ref{system-a-b})
is uniquely solvable. We shall now consider the invertibility of the
coefficient matrix $A(\Omega,\epsilon)$ of the linear system (\ref{system-a-b})
in various domains in the $\Omega$-plane for small $\epsilon>0$.
Figure \ref{fig:omega} shows schematically the location of these
domains on the $\Omega$-plane.

\begin{figure}
\begin{centering}
\includegraphics[width=0.5\textwidth]{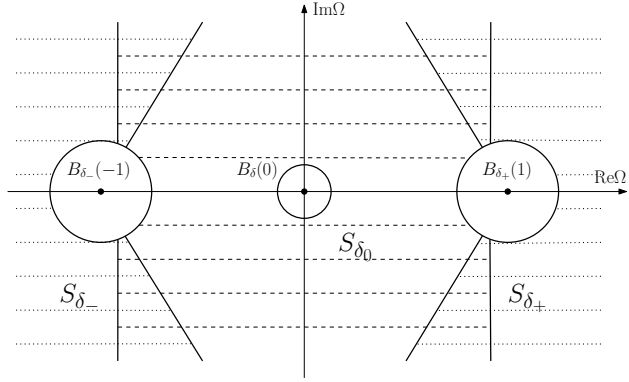} 
\par\end{centering}

\caption{Schematic display of various domains in the $\Omega$-plane.}

\label{fig:omega} 
\end{figure}

Fix $\delta_{0}\in(0,1)$. Let $\Omega$ belong to the vertical strip
\[
S_{\delta_{0}}:=\left\{ \Omega\in\C:\quad{\rm Re}(\Omega)\in[-\delta_{0},\delta_{0}]\right\} .\]
 Then $z(\lambda_{\pm})=-i\kappa_{\pm}$ are uniquely determined from
the equation \[
e^{\kappa_{\pm}}+e^{-\kappa_{\pm}}-2=\frac{1\mp\Omega}{\epsilon},\quad{\rm Re}(\kappa_{\pm})>0,\quad{\rm Im}(\kappa_{\pm})\in[-\pi,\pi),\]
 which admits the asymptotic expansion \[
e^{\kappa_{\pm}}=\frac{1\mp\Omega}{\epsilon}+2-\frac{\epsilon}{1\mp\Omega}+{\cal O}(\epsilon^{2})\quad\mbox{{\rm as}}\quad\epsilon\to0\]
 and \[
e^{-\kappa_{\pm}}=\frac{\epsilon}{1\mp\Omega}+{\cal O}(\epsilon^{2})\quad\mbox{{\rm as}}\quad\epsilon\to0.\]

Therefore, both $\epsilon\sinh(\kappa_{\pm})$ and $Q^{\pm}(\Omega,\epsilon)$
are analytic in $\epsilon$ near $\epsilon=0$ and \[
2i\epsilon\sin z(\lambda_{\pm})=2\epsilon\sinh(\kappa_{\pm})=1\mp\Omega+2\epsilon+{\cal O}(\epsilon^{2})\quad\mbox{{\rm as}}\quad\epsilon\to0\]
 and \[
Q^{\pm}(\Omega,\epsilon)=I+{\cal O}(\epsilon)\quad\mbox{{\rm as}}\quad\epsilon\to0.\]

It becomes now clear that $A(\Omega,\epsilon)$ is analytic in $\Omega\in S_{\delta_{0}}$
and $\epsilon\in(-\epsilon_{0},\epsilon_{0})$ with the limit \begin{equation}
A(\Omega,0)=\left[\begin{array}{cc}
-(p+\Omega)I & -pI\\
-pI & -(p-\Omega)I\end{array}\right].\label{matrix-A-zero}\end{equation}

Matrix $A(\Omega,0)\in\C^{2N\times2N}$ is singular only for $\Omega=0$.
Thanks to analyticity of $A(\Omega,\epsilon)$, the determinant $D(\Omega,\epsilon)={\rm det}A(\Omega,\epsilon)$
is also analytic in these variables and \[
D(\Omega,\epsilon)=(-\Omega^{2})^{N}+{\cal O}(\epsilon)\quad\mbox{{\rm as}}\quad\epsilon\to0.\]
 Therefore, there exist $2N$ zeros of $D(\Omega,\epsilon)$ for small
$\epsilon\in(0,\epsilon_{0})$ in a small disk $B_{\delta}(0)$ with
$\delta={\cal O}(\epsilon^{1/2N})$. By Cramer's rule, these zeros
of $D(\Omega,\epsilon)$ give poles of $R_{L}(\Omega)$.

Fix $\delta_{+}\in(0,1)$ and $\theta_{+}\in(\frac{\pi}{2},\pi)$.
We now consider $\Omega$ in the domain \[
S_{\delta_{+}}:=\left\{ \Omega=1+re^{i\theta},\quad r>\delta_{+},\;\;\theta\in(-\theta_{+},\theta_{+})\right\} .\]
 In this domain, we have the same presentation for $z(\lambda_{-})=-i\kappa_{-}$
but a different presentation for $z(\lambda_{+})=-i\kappa_{+}-\pi$.
Now $\kappa_{+}$ is uniquely determined from the equation \[
e^{\kappa_{+}}+e^{-\kappa_{+}}+2=\frac{\Omega-1}{\epsilon}=\frac{r}{\epsilon}e^{i\theta},\quad{\rm Re}(\kappa_{+})>0,\quad{\rm Im}(\kappa_{+})\in[0,2\pi),\]
 which admits the asymptotic expansions \[
e^{\kappa_{+}}=\frac{\Omega-1}{\epsilon}-2-\frac{\epsilon}{\Omega-1}+{\cal O}(\epsilon^{2})\quad\mbox{{\rm as}}\quad\epsilon\to0\]
 and \[
2i\epsilon\sin z(\lambda_{+})=1-\Omega+2\epsilon+{\cal O}(\epsilon^{2})\quad\mbox{{\rm as}}\quad\epsilon\to0.\]

Since ${\rm Re}(\kappa_{+})\to\infty$ as $\epsilon\to0$, $A(\Omega,0)$
is the same as matrix (\ref{matrix-A-zero}) and it is invertible
for $\Omega\in S_{\delta_{+}}$. Similar arguments can be developed
for \[
S_{\delta_{-}}:=\left\{ \Omega=-1+re^{i\theta},\quad r>\delta_{-},\;\;\theta\in(\theta_{-},2\pi-\theta_{-})\right\} ,\]
 where $\delta_{-}\in(0,1)$ and $\theta_{-}\in(0,\frac{\pi}{2})$.
Because there are choices of $\delta_{0},\delta_{\pm}>0$ such that
\[
S_{\delta_{0}}\cup S_{\delta_{+}}\cup S_{\delta_{-}}=\mathbb{C}\backslash\{B_{\delta_{+}}(1)\cup B_{\delta_{-}}(-1)\},\]
 we obtain the assertion of the lemma. \end{proof}

\begin{remark} The proof of Lemma \ref{lemma-first-part} implies
that poles of $R_{L}(\Omega)$ may have size $|\Omega|={\cal O}(\epsilon^{1/2N})$.
The results of the perturbation expansions (see \cite{PKF1} for details)
imply that the eigenvalues bifurcating from $0$ in the full spectral
problem (\ref{spectral-stability}) have size ${\cal O}(\epsilon^{1/2})$.
Moreover, the same perturbation expansion technique can be applied
to show that eigenvalues of the truncated spectral problem (\ref{equation-resolvent-1})
have the same size ${\cal O}(\epsilon^{1/2})$. \end{remark}

\subsection{Resolvent inside the continuous spectrum}

We shall now consider the resolvent operator $R_{L}(\Omega)$ inside
the continuous spectrum \[
\sigma_{c}(L):=[1,1+4\epsilon]\cup[-1-4\epsilon,-1].\]
 Thanks to the symmetry of system (\ref{resolvent-a-b})--(\ref{system-a-b})
in $\Omega$, we can consider only one branch of the continuous spectrum
$[1,1+4\epsilon]$. Therefore, we set $\Omega=1+\epsilon\omega$ with
$\omega\in[0,4]$ and define \[
z(\lambda_{+})=z(\omega)\equiv\theta\quad\mbox{{\rm and}}\quad z(\lambda_{-})=z(-2\epsilon^{-1}-\omega)\equiv-i\kappa.\]
 It follows from (\ref{transcendental-equation-1}) and (\ref{transcendental-equation-2})
that $\theta\in[-\pi,0]$ and $\kappa>0$ are uniquely defined from
equations \begin{equation}
2-2\cos(\theta)=\omega,\quad2\epsilon(\cosh(\kappa)-1)=2+\epsilon\omega,\quad\omega\in[0,4].\label{transcendental-equation-3}\end{equation}

The choice of $\theta\in[-\pi,0]$ corresponds to the limiting operator
$R_{0}^{+}(\omega)$ of the free resolvent. Since $R_{0}^{+}(\omega):l_{\sigma}^{2}(\Z)\to l_{-\sigma}^{2}(\Z)$
is well defined for $\omega\in(0,4)$ and $\sigma>\frac{1}{2}$, $R_{L}^{+}(1+\epsilon\omega)$
is a bounded map from $l_{\sigma}^{2}(\Z)\times l_{\sigma}^{2}(\Z)$
to $l_{-\sigma}^{2}(\Z)\times l_{-\sigma}^{2}(\Z)$ for any $\omega\in(0,4)$
and $\sigma>\frac{1}{2}$ if and only if there exists a unique solution
of the linear system (\ref{system-a-b}). On the other hand, the free
resolvent is singular in the limits $\omega\downarrow0$ and $\omega\uparrow4$
and, therefore, we need to be careful in solving system (\ref{resolvent-a-b})--(\ref{system-a-b})
in this limit.

The main result of this section is given by the following theorem.

\begin{theorem} \label{theorem-continuous-spectrum} Let $m_{1}=m_{2}=\cdots m_{N-1}=1$.
There exists $\epsilon_{0}>0$ such that for any $\omega\in[0,4]$
and any $\epsilon\in(0,\epsilon_{0})$, there exist $C>0$ such that
\begin{equation}
\|R_{L}^{+}(1+\epsilon\omega)\|_{l_{1}^{1}\times l_{1}^{1}\to l^{\infty}\times l^{\infty}}\leq C\epsilon^{-1},\label{bound-resolvent-l1-1}\end{equation}
 where the upper sign indicates that $\omega$ is parameterized by
$\omega=2-2\cos(\theta)$ for $\theta\in[-\pi,0]$. \end{theorem}

To prove Theorem \ref{theorem-continuous-spectrum}, we analyze solutions
of system (\ref{system-a-b}) for $\omega\in[0,4]$. Let us rewrite
explicitly \[
q_{j}^{+}=e^{-im_{j}\theta}\quad{\rm and}\quad q_{j}^{-}=e^{-m_{j}\kappa},\quad j\in\{1,2,...,N-1\}.\]
 The coefficient matrix (\ref{coefficient-matrix-1}) for $\Omega=1+\epsilon\omega$
with $\omega\in[0,4]$ is rewritten in the form \begin{equation}
A(\theta,\epsilon)\equiv\left[\begin{array}{cc}
2i\epsilon\sin(\theta)I-(1+p)M(\theta) & -pM(\theta)\\
-pN(\kappa) & 2\epsilon\sinh(\kappa)I-(1+p)N(\kappa)\end{array}\right],\label{coefficient-matrix-2}\end{equation}
 where $M(\theta)\equiv Q(q_{1}^{+},q_{2}^{+},\cdots,q_{N-1}^{+})$
and $N(\kappa)\equiv Q(q_{1}^{-},q_{2}^{-},\cdots,q_{N-1}^{-})$.
Note that $\theta$ and $M(\theta)$ are $\epsilon$-independent,
whereas $N(\kappa)$ depends on $\epsilon$ via $\kappa$. The linear
system (\ref{system-a-b}) is now expressed in the matrix form \begin{equation}
A(\theta,\epsilon)c=h(\theta,\epsilon),\label{abstract-linear-system}\end{equation}
 where components of $c\in\C^{2N}$ and $h\in\C^{2N}$ are given by
\[
\left\{ \begin{array}{l}
a_{n}\\
b_{n}\end{array}\right\} _{n\in U_{+}\cup U_{-}}\quad\mbox{{\rm and}}\quad\left\{ \begin{array}{l}
\sum_{m\in\Z}e^{-i\theta|n-m|}f_{m}\\
\sum_{m\in\Z}e^{-\kappa|n-m|}g_{m}\end{array}\right\} _{n\in U_{+}\cup U_{-}}.\]

Thanks to the asymptotic expansion \[
e^{\kappa}=\frac{2}{\epsilon}+2+\omega-\frac{\epsilon}{2}+{\cal O}(\epsilon^{2})\quad\mbox{{\rm as}}\quad\epsilon\to0,\]
 we have \[
2\epsilon\sinh(\kappa)=2+(2+\omega)\epsilon+{\cal O}(\epsilon^{2})\quad\mbox{{\rm as}}\quad\epsilon\to0.\]
 Both $A(\theta,\epsilon)$ and $h(\theta,\epsilon)$ are analytic
in $\theta\in[-\pi,0]$ and $\epsilon\in(-\epsilon_{0},\epsilon_{0})$.
The following lemma establishes the invertibility condition for matrix
$A(\theta,\epsilon)$.

\begin{lemma} For any $\epsilon\in(0,\epsilon_{0})$, matrix $A(\theta,\epsilon)$
has a zero eigenvalue of geometric and algebraic multiplicities $N-1$
for $\theta=-\pi$ and $\theta=0$. If $m_{1}=m_{2}=\cdots=m_{N-1}=1$,
matrix $A(\theta,\epsilon)$ is invertible for any $\theta\in(-\pi,0)$.
\label{lemma-A} \end{lemma}

\begin{proof} We use the fact that matrix $A(\theta,\epsilon)$ is
analytic in $\epsilon$ for small $\epsilon\in(-\epsilon_{0},\epsilon_{0})$.
Therefore, it remains invertible if $A(\theta,0)$ is invertible.
To consider the limit $\epsilon\to0$, we note that $\kappa\to\infty$
and $N(\kappa)\to I$ as $\epsilon\to0$, so we have \[
A(\theta,0)=\left[\begin{array}{cc}
-(1+p)M(\theta) & -pM(\theta)\\
-pI & (1-p)I\end{array}\right].\]

For any $p\in\N$, matrix $A(\theta,0)$ is invertible if and only
if matrix $M(\theta)$ is invertible. Let us then compute \[
D_{N}(q_{1},q_{2},\cdots,q_{N-1}):={\rm det}Q(q_{1},q_{2},\cdots,q_{N-1}).\]
 We note that $D_{N}(\pm1,q_{2},\cdots,q_{N-1})=0$ and $D_{N}(q_{1},q_{2},\cdots,q_{N-1})$
is a quadratic polynomial of $q_{1}$. Therefore, \[
D_{N}(q_{1},q_{2},\cdots,q_{N-1})=(1-q_{1}^{2})D_{N}(0,q_{2},\cdots,q_{N-1})=(1-q_{1}^{2})D_{N-1}(q_{2},\cdots,q_{N-1}).\]
 Continuing the expansion recursively, we obtain the exact formula
\begin{equation}
D_{N}(q_{1},q_{2},\cdots,q_{N-1})=(1-q_{1}^{2})(1-q_{2}^{2})\cdots(1-q_{N-1}^{2}),\label{determinant-D-N}\end{equation}
 from which we conclude that $Q(q_{1},q_{2},\cdots,q_{N-1})$ is invertible
if and only if all $q_{j}\neq\pm1$. This implies that $M(\theta)$
is invertible if and only if all $e^{-im_{j}\theta}\neq\pm1$, which
is satisfied if all $m_{j}=1$ and $\theta\in(-\pi,0)$. The second
assertion of the lemma is proved: for any $\epsilon\in[0,\epsilon_{0})$,
matrix $A(\theta,\epsilon)$ is invertible for $\theta\in(-\pi,0)$
if all $m_{j}=1$.

The first assertion of the lemma tells us that for any $\epsilon\in(0,\epsilon_{0})$,
matrices $A_{+}(\epsilon):=A(0,\epsilon)$ and $A_{-}(\epsilon):=A(-\pi,\epsilon)$
have a zero eigenvalue of geometric and algebraic multiplicities $N-1$.
We write $A_{\pm}(\epsilon)$ explicitly in the form \[
A_{\pm}(\epsilon)=\left[\begin{array}{cc}
-(1+p)M_{\pm} & -pM_{\pm}\\
-pN(\kappa_{\pm}) & 2\epsilon\sinh(\kappa_{\pm})I-(1+p)N(\kappa_{\pm})\end{array}\right],\]
 where $\kappa_{\pm}>0$ are uniquely defined by \[
2\epsilon(\cosh(\kappa_{+})-1)=2,\quad2\epsilon(\cosh(\kappa_{-})-1)=2+4\epsilon,\]
 whereas matrices $M_{\pm}$ are given by \[
M_{+}=\left[\begin{array}{ccccc}
1 & 1 & 1 & \cdots & 1\\
1 & 1 & 1 & \cdots & 1\\
1 & 1 & 1 & \cdots & 1\\
\vdots & \vdots & \vdots & \vdots & \vdots\\
1 & 1 & 1 & \cdots & 1\end{array}\right]\]
 and \[
M_{-}=\left[\begin{array}{ccccc}
1 & (-1)^{m_{1}} & (-1)^{m_{1}+m_{2}} & \cdots & (-1)^{m_{1}+m_{2}+...+m_{N-1}}\\
(-1)^{m_{1}} & 1 & (-1)^{m_{2}} & \cdots & (-1)^{m_{2}+...+m_{N-1}}\\
(-1)^{m_{1}+m_{2}} & (-1)^{m_{2}} & 1 & \cdots & (-1)^{m_{3}+...+m_{N-1}}\\
\vdots & \vdots & \vdots & \vdots & \vdots\\
(-1)^{m_{1}+m_{2}+...+m_{N-1}} & (-1)^{m_{2}+...+m_{N-1}} & (-1)^{m_{3}+...+m_{N-1}} & \cdots & 1\end{array}\right].\]
 It is clear that ${\rm Null}(M_{+})$ and ${\rm Null}(M_{-})$ are
$(N-1)$-dimensional.

The first $N$ rows of $A_{+}(\epsilon)$ are identical to the first
row, whereas the last $N$ rows of $A_{+}(\epsilon)$ are linearly
independent at $\epsilon=0$ and, by continuity, for small $\epsilon\in[0,\epsilon_{0})$.
Therefore, ${\rm Null}(A_{+}(\epsilon))$ is $(N-1)$-dimensional
for any $\epsilon\in[0,\epsilon_{0})$. Similarly, the second, third,
and $N$-th rows of $A_{-}(\epsilon)$ are identical to the first
row multiplied by $(-1)^{m_{1}}$, $(-1)^{m_{1}+m_{2}}$, and $(-1)^{m_{1}+m_{2}+...+m_{N-1}}$
respectively. The last $N$ rows of $A_{-}(\epsilon)$ are linearly
independent for small $\epsilon\geq0$. Therefore, ${\rm Null}(A_{-}(\epsilon))$
is $(N-1)$-dimensional for any $\epsilon\in[0,\epsilon_{0})$.

It remains to prove that the zero eigenvalue of $A_{\pm}(\epsilon)$
is not degenerate (has equal geometric and algebraic multiplicity)
for $\epsilon\in(0,\epsilon_{0})$. It is clear from the explicit
form of $A_{\pm}(0)$ and $M_{\pm}$ that \begin{equation}
u\in{\rm Null}(A_{\pm}(0))\quad\Leftrightarrow\quad u=\left[\begin{array}{c}
(1-p)w\\
pw\end{array}\right],\quad w\in{\rm Null}(M_{\pm}).\label{relation-null-space}\end{equation}
 To construct a generalized kernel, we consider the inhomogeneous
equation \[
A_{\pm}(0)\tilde{u}=u,\quad u\in{\rm Null}(A_{\pm}(0)).\]
 Then, we obtain for $w\in{\rm Null}(M_{\pm})$, \[
\tilde{u}=\left[\begin{array}{c}
(1-p)\tilde{w}-w\\
p\tilde{w}\end{array}\right],\quad M_{\pm}\tilde{w}=(p-1)w.\]
 If $p\neq1$, then no $\tilde{w}\in\C^{N}$ exists because $M_{\pm}$
is symmetric. Therefore, for $p\neq1$, the zero eigenvalue has equal
geometric and algebraic multiplicity for the matrix $A_{\pm}(0)$
and, by continuity, for the matrix $A_{\pm}(\epsilon)$ for $\epsilon\in[0,\epsilon_{0})$.

The case $p=1$ needs a separate consideration since $\tilde{w}=0$
and the zero eigenvalue of $A_{\pm}(0)$ has geometric multiplicity
$N-1$ and algebraic multiplicity $2N-2$. This case is considered
in Appendix A, where we show that the degeneracy is broken for any
$\epsilon\neq0$, so that $A_{\pm}(\epsilon)$ in the case $p=1$
still has a zero eigenvalue of equal geometric and algebraic multiplicity
$N-1$ for any $\epsilon\in(0,\epsilon_{0})$. \end{proof}

Because the coefficient matrix $A(\theta,\epsilon)$ is singular at
$\theta=0$ and $\theta=-\pi$, we shall consider the limiting behavior
of solutions of the linear system (\ref{abstract-linear-system})
near these points. The following abstract lemma gives the sufficient
condition that the unique solution $c$ of the linear system (\ref{abstract-linear-system})
for small $\theta\neq0$ and fixed $\epsilon\in(0,\epsilon_{0})$
remains bounded in the limit $\theta\to0$. Because $\epsilon$ is
fixed, we can drop this parameter from the notations of the lemma.

\begin{lemma} \label{lemma-fredholm} Assume that $A(\theta)\in\C^{2N\times2N}$
and $h(\theta)\in\C^{2N}$ are analytic in $\theta\in(-\theta_{0},\theta_{0})$
for $\theta_{0}>0$ and consider solutions of \[
A(\theta)c=h(\theta),\quad c\in\C^{2N}.\]
 Assume that $A(\theta)$ is invertible for $\theta\neq0$ and singular
for $\theta=0$ and that the zero eigenvalue of $A(0)$ has equal
geometric and algebraic multiplicity $n\leq2N$. A unique solution
$c$ for $\theta\neq0$ is bounded as $\theta\to0$ if \begin{equation}
h(0)\perp{\rm Null}(A^{*}(0))\quad{\rm and}\quad{\rm Null}(A'(0)|_{{\rm Null}(A(0))})=\{0\}.\label{condition-bounded}\end{equation}
 \end{lemma}

\begin{remark} We denote the Hermite conjugate of a matrix $A_{0}\in\C^{2N\times2N}$
by $A_{0}^{*}=\overline{A_{0}^{T}}$. Let \begin{equation}
{\rm Null}(A_{0})={\rm span}\{u_{1},...,u_{n}\}\quad{\rm and}\quad{\rm Null}(A_{0}^{*})={\rm span}\{v_{1},...,v_{n}\},\label{basis-1}\end{equation}
 where $\{u_{1},...,u_{n}\}$ and $\{v_{1},...,v_{n}\}$ are mutually
orthogonal bases, so that \begin{equation}
\langle u_{i},v_{j}\rangle_{\C^{2N}}=\delta_{i,j}\quad\mbox{{\rm for all}}\quad1\leq i,j\leq n.\label{basis-2}\end{equation}
 The restriction of matrix $A_{1}\in\C^{2N\times2N}$ on ${\rm Null}(A_{0})$
denoted by $A_{1}|_{{\rm Null}(A_{0})}$ can be expressed by the matrix
$P\in\C^{n\times n}$ with elements \begin{equation}
P_{ij}=\langle v_{j},A_{1}u_{i}\rangle_{\C^{2N}}\quad\mbox{{\rm for all}}\quad1\leq i,j\leq n.\label{basis-3}\end{equation}
 \end{remark}

\begin{proof} The proof of the lemma is achieved with the method
of Lyapunov--Schmidt reductions. Using analyticity of $A(\theta)$
and $h(\theta)$, let us expand \[
A(\theta)=A_{0}+\theta A_{1}+\theta^{2}\tilde{A}(\theta),\quad h(\theta)=h_{0}+\theta h_{1}+\theta^{2}\tilde{h}(\theta),\]
 where $A_{0}=A(0)$, $A_{1}=A'(0)$, $h_{0}=h(0)$, $h_{1}=h'(0)$,
and $\tilde{A}(\theta)$ and $\tilde{h}(\theta)$ are bounded as $\theta\to0$.
Given the basis for ${\rm Null}(A_{0})$ in (\ref{basis-1}), we consider
the orthogonal decomposition of the solution \begin{equation}
c=\sum_{j=1}^{n}a_{j}u_{j}+b,\quad b\perp{\rm Null}(A_{0}).\label{linear-system-0}\end{equation}
 The linear system becomes \begin{equation}
(A_{0}+\theta A_{1}+\theta^{2}\tilde{A}(\theta))b+\theta\sum_{j=1}^{n}a_{j}(A_{1}+\theta\tilde{A}(\theta))u_{j}=h_{0}+\theta h_{1}+\theta^{2}\tilde{h}(\theta).\label{linear-system-1}\end{equation}
 Projections of system (\ref{linear-system-1}) to the basis for ${\rm Null}(A_{0}^{*})$
in (\ref{basis-1}) give $n$ equations \begin{equation}
\sum_{j=1}^{n}\left(P_{ij}+\theta\tilde{P}_{ij}(\theta)\right)a_{j}+\langle v_{i},(A_{1}+\theta\tilde{A}(\theta))b\rangle_{\C^{2N}}=\langle v_{i},h_{1}+\theta\tilde{h}(\theta)\rangle_{\C^{2N}},\quad1\leq i\leq n,\label{linear-system-2}\end{equation}
 where $P_{ij}$ is given in (\ref{basis-3}), $\tilde{P}_{ij}(\theta)=\langle v_{i},\tilde{A}(\theta)u_{j}\rangle_{\C^{2N}}$
is bounded as $\theta\to0$, and we have used the condition $h_{0}\perp{\rm Null}(A_{0}^{*})$.

Let $Q:\C^{2N}\to{\rm Ran}(A_{0})\subset\C^{2N}$ and $Q^{*}:\C^{2N}\to{\rm Ran}(A_{0}^{*})\subset\C^{2N}$
be the projection operators. Recall that ${\rm Ran}(A_{0})\perp{\rm Null}(A_{0}^{*})$
and ${\rm Ran}(A_{0}^{*})\perp{\rm Null}(A_{0})$. Projection of system
(\ref{linear-system-1}) to ${\rm Ran}(A_{0})$ gives an equation
for $b$ \begin{equation}
Q(A_{0}+\theta A_{1}+\theta^{2}\tilde{A}(\theta))Q^{*}b=Q(h_{0}+\theta h_{1}+\theta^{2}\tilde{h}(\theta))-\sum_{j=1}^{n}a_{j}Q(A_{1}+\theta\tilde{A}(\theta))u_{j}.\label{linear-system-3}\end{equation}
 Because $QA_{0}Q^{*}$ is invertible, there is a unique map $\C^{n}\ni(a_{1},...,a_{n})\mapsto b\in{\rm Ran}(A_{0}^{*})$
for any $\theta\in(-\theta_{0},\theta_{0})$ such that $b$ is a solution
of system (\ref{linear-system-3}) and for any $\theta\in(-\theta_{0},\theta_{0})$,
there is $C>0$ such that \begin{equation}
\|b-Q^{*}A_{0}^{-1}Qh_{0}\|_{\C^{2N}}\leq C\theta.\label{linear-system-3a}\end{equation}

Since ${\rm Null}(A_{1}|_{{\rm Null}(A_{0})})=\{0\}$, matrix $P$
is invertible. For any $b$ from solution of system (\ref{linear-system-3})
satisfying bound (\ref{linear-system-3a}), there exists a unique
solution of system (\ref{linear-system-2}) for $(a_{1},...,a_{n})$
for any $\theta\in(-\theta_{0},\theta_{0})$ such that \begin{equation}
\exists C>0:\quad\|a-P^{-1}(I-Q)(h_{1}-A_{1}Q^{*}A_{0}^{-1}Qh_{0})\|_{\C^{n}}\leq C\theta.\label{linear-system-3b}\end{equation}

For any $\theta\neq0$, the solution of system $A(\theta)c=h(\theta)$
is unique. Therefore, the unique solution obtained from the decomposition
(\ref{linear-system-0}) for any $\theta\in(-\theta_{0},\theta_{0})$
is equivalent to the unique solution of system $A(\theta)c=h(\theta)$
for $\theta\neq0$. \end{proof}

We shall check that the conditions (\ref{condition-bounded}) of Lemma
\ref{lemma-fredholm} are satisfied for our particular matrix $A(\theta,\epsilon)$
and the right-hand-side vector $h(\theta,\epsilon)$ for both end
points $\theta=0$ and $\theta=-\pi$.

\begin{lemma} Let $h_{+}(\epsilon):=h(0,\epsilon)$ and $h_{-}(\epsilon):=h(-\pi,\epsilon)$.
For any $\epsilon\in(0,\epsilon_{0})$, it is true that \begin{equation}
h_{\pm}(\epsilon)\perp{\rm Null}(A_{\pm}^{*}(\epsilon))\quad{\rm and}\quad{\rm Null}(\partial_{\theta}A_{\pm}(\epsilon)|_{{\rm Null}(A_{\pm}(\epsilon))})=\{0\}.\label{condition-bounded-lemma}\end{equation}
 \label{lemma-condition-bounded} \end{lemma}

\begin{proof} It is sufficient to develop the proof for $\theta=0$.
The proof for $\theta=-\pi$ is similar.

Recall that the first $N$ rows of $A(0,\epsilon)$ are identical
to the first row. Since components of $h(0,\epsilon)$ are given by
\[
\left\{ \begin{array}{l}
\sum_{m\in\Z}f_{m}\\
\sum_{m\in\Z}e^{-\kappa|n-m|}g_{m}\end{array}\right\} _{n\in U_{+}\cup U_{-}},\]
 the first $N$ entries of $h(0,\epsilon)$ are also identical so
that $h(0,\epsilon)\in{\rm Ran}(A(0,\epsilon))\perp{\rm Null}(A^{*}(0,\epsilon))$
for any $\epsilon\in(0,\epsilon_{0})$. Therefore, the first condition
(\ref{condition-bounded-lemma}) is satisfied.

Next, we compute $A_{1}(\epsilon)=\partial_{\theta}A(\theta,\epsilon)|_{\theta=0}$.
We know that \[
2\epsilon(\cosh(\kappa)-1)=2+\epsilon(2-2\cos(\theta))\quad\Rightarrow\quad\frac{d\kappa}{d\theta}=\frac{\sin(\theta)}{\sinh(\kappa)},\]
 therefore, \begin{equation}
A_{1}(\epsilon)\equiv i\left[\begin{array}{cc}
2\epsilon I+(1+p)R & pR\\
0 & 0\end{array}\right],\label{matrix-A-1}\end{equation}
 where \[
R=\left[\begin{array}{ccccc}
0 & m_{1} & m_{1}+m_{2} & \cdots & m_{1}+m_{2}+\cdots+m_{N-1}\\
m_{1} & 0 & m_{2} & \cdots & m_{2}+m_{3}+\cdots+m_{N-1}\\
m_{1}+m_{2} & m_{2} & 0 & \cdots & m_{3}+\cdots+m_{N-1}\\
\vdots & \vdots & \vdots & \vdots & \vdots\\
m_{1}+m_{2}+\cdots+m_{N-1} & m_{2}+\cdots+m_{N-1} & m_{3}+\cdots+m_{N-1} & \cdots & 0\end{array}\right].\]

Let $P(\epsilon)$ be the matrix in $\C^{(N-1)\times(N-1)}$ which
represents the restriction $A_{1}(\epsilon)|_{{\rm Null}(A_{0}(\epsilon))}$.
Existence of $a\in{\rm Null}(P(\epsilon))\subset\C^{N-1}$ is equivalent
to existence of $u\in{\rm Null}(A_{0}(\epsilon))\subset\C^{2N}$ such
that $A_{1}(\epsilon)u\in{\rm Ran}(A_{0}(\epsilon))\perp{\rm Null}(A_{0}^{*}(\epsilon))$.
In other words, we need to find $u\in{\rm Null}(A_{0}(\epsilon))$
such that the first $N$ entries of $A_{1}(\epsilon)u$ are identical
(the other $N$ entries of $A_{1}(\epsilon)u$ are zeros).

By continuity in $\epsilon\in[0,\epsilon_{0})$, the second condition
(\ref{condition-bounded-lemma}) is satisfied if it is satisfied for
$\epsilon=0$. Therefore, it is sufficient to check the existence
of $u\in{\rm Null}(A_{0}(0))$ such that the first $N$ entries of
$A_{1}(0)u$ are identical.

It follows from relations (\ref{relation-null-space}) and (\ref{matrix-A-1})
that existence of $u\in{\rm Null}(A_{0}(0))$ such that the first
$N$ entries of $A_{1}(0)u$ are identical is equivalent to the existence
of $w\in{\rm Null}(M_{+})\subset\C^{N}$ such that all entries of
$Rw$ are identical.

If $w=[w_{1},w_{2},...,w_{N}]^{T}\in{\rm Null}(M_{+})$, then \begin{equation}
w_{1}+w_{2}+...+w_{N}=0.\label{constraint-sum}\end{equation}
 Condition $(Rw)_{1}=(Rw)_{2}$ gives \[
m_{1}(w_{2}+...+w_{N})=m_{1}w_{1}.\]
 Constraint (\ref{constraint-sum}) implies that if $m_{1}\neq0$,
then $w_{1}=0$ and $w_{2}+\cdots+w_{N}=0$. Continuing by induction
for condition $(Rw)_{j}=(Rw)_{j+1}$, where $j\in\{1,2,...,N-1\}$,
we obtain that if $m_{j}\neq0$, then $w_{j}=0$ for all $j\in\{1,2,...,N-1\}$.
In view of constraint (\ref{constraint-sum}), we have $w_{N}=0$
that is $w=0\in\C^{N}$. As a result, we have proved that ${\rm Null}(A_{1}(0)|_{{\rm Null}(A_{0}(0))})=\{0\}$.
By continuity in $\epsilon\in[0,\epsilon_{0})$, ${\rm Null}(A_{1}(\epsilon)|_{{\rm Null}(A_{0}(\epsilon))})=\{0\}$
for small $\epsilon\neq0$, which gives the second condition (\ref{condition-bounded-lemma})
for $\theta=0$. \end{proof}

\begin{remark} \label{remark-lemma-4} Lemma \ref{lemma-condition-bounded}
is proved without assuming that all $m_{j}=1$. \end{remark}

\begin{proof1}\textbf{of Theorem \ref{theorem-continuous-spectrum}.}
By Lemma \ref{lemma-condition-bounded}, assumptions of Lemma \ref{lemma-fredholm}
are satisfied and the unique solution of system (\ref{abstract-linear-system})
for $\theta\in(-\pi,0)$ is continued to the unique bounded limit
$c_{0}=\lim_{\theta\to0}c$. From the first $N$ equations of system
(\ref{system-a-b}), we infer that \[
\theta=0:\quad\sum_{m\in U_{+}\cup U_{-}}((1+p)a_{m}+pb_{m})=-\sum_{m\in\Z}f_{m}.\]
 As a result, the simple pole singularity at $\theta=0$ ($z(\lambda_{+})=0$)
in the Green's function representation (\ref{resolvent-a-b}) with
the Puiseux expansion (\ref{Puiseux-free}) is canceled. Similarly,
the simple pole singularity at $\theta=-\pi$ is cancelled. On the
other hand, the representation (\ref{resolvent-a-b}) contains $\epsilon$
in the denominator, which does not cancel out generally. As a result,
Lemma \ref{lemma-A} for all $m_{j}=1$ and Lemma \ref{lemma-condition-bounded}
give that for any $\omega\in[0,4]$ and any $\epsilon\in(0,\epsilon_{0})$,
there exists $C>0$ such that \[
\|{\bf a}\|_{l^{\infty}}\leq C\epsilon^{-1}.\]
 This gives bound (\ref{bound-resolvent-l1-1}) and hence Theorem
\ref{theorem-continuous-spectrum}. \end{proof1}

\subsection{Matching conditions for the resolvent operator}

To complete the proof of Theorem \ref{theorem-resolvent}, we need
to prove that no singularities of linear system (\ref{system-a-b})
are located inside the disks $B_{\delta_{+}}(1)$ and $B_{\delta_{-}}(-1)$
for $\epsilon$-independent $\delta_{\pm}>0$. It is again sufficient
to consider the disk $B_{\delta_{+}}(1)$ because of the symmetry
in the $\Omega$-plane.

The free resolvent operator $R_{0}^{+}(\lambda):l_{\sigma}^{2}(\Z)\to l_{-\sigma}^{2}(\Z)$
with $\sigma>\frac{1}{2}$ is extended meromorphically in variable
$\theta(\lambda)$ for $\lambda\in\C^{+}\backslash[0,4]$ with simple
poles at $\theta=0$ ($\lambda=0$) and $\theta=-\pi$ ($\lambda=4$).
By Theorem \ref{theorem-continuous-spectrum}, the resolvent operator
$R_{L}^{+}(1+\epsilon\omega):l_{1}^{1}(\Z)\times l_{1}^{1}(\Z)\to l^{\infty}(\Z)\times l^{\infty}(\Z)$
is bounded for $\omega\in[0,4]$ and the pole singularities are canceled.
As a result, the resolvent operator $R_{L}^{+}(1+\epsilon\lambda)$
can be extended as a bounded operator from $l_{\sigma}^{2}(\Z)\times l_{\sigma}^{2}(\Z)$
to $l_{-\sigma}^{2}(\Z)\times l_{-\sigma}^{2}(\Z)$ with $\sigma>\frac{1}{2}$
for any $\lambda\in\C^{+}\backslash[0,4]$. We need to show that no
singularities of the resolvent operator $R_{L}(1+\epsilon\lambda)$
exist in the upper semi-annulus \[
D_{\delta_{+}}=\left\{ \lambda\in\C^{+}:\quad\gamma_{+}<|\lambda|<\delta_{+}\epsilon^{-1}\right\} \subset B_{\delta_{+}}(1),\]
 where $\gamma_{+}>4$ and $\delta_{+}\in(0,1)$. A similar analysis
can also be used to show that the resolvent operator $R_{L}^{-}(1+\epsilon\lambda)$
can be extended as a bounded operator in the lower semi-disk in $B_{\delta_{+}}(1)$.

\begin{lemma} \label{lemma-semi-annulus} For any $\epsilon\in(0,\epsilon_{0})$
and all $\lambda\in D_{\delta_{+}}$, the resolvent operator $R_{L}(1+\epsilon\lambda)$
is a bounded operator from $l^{2}(\Z)\times l^{2}(\Z)$ to $l^{2}(\Z)\times l^{2}(\Z)$.
\end{lemma}

\begin{proof} Since the continuous spectrum does not touch boundaries
of $D_{\delta_{+}}$, the statement is true if and only if there exists
a unique solution of linear system (\ref{system-a-b}).

Let us denote $z(\lambda_{+})=z(\lambda)$ and $z(\lambda_{-})=-i\kappa(\lambda)$,
where $z(\lambda)$ is found from the transcendental equation (\ref{transcendental-equation-1})
and $\kappa(\lambda)$ with ${\rm Re}(\kappa(\lambda))>0$ admits
the asymptotic expansion for $\lambda\in D_{\delta_{+}}$ \[
e^{\kappa(\lambda)}=\frac{2+\epsilon\lambda}{\epsilon}+2-\frac{\epsilon}{2+\epsilon\lambda}+{\cal O}(\epsilon^{2})\quad\mbox{{\rm as}}\quad\epsilon\to0.\]
 As earlier, we denote $q_{j}^{+}=e^{-im_{j}z(\lambda)}$ and $q_{j}^{-}=e^{-m_{j}\kappa(\lambda)}$
for $j\in\{1,2,...,N-1\}$.

We write the coefficient matrix (\ref{coefficient-matrix-1}) for
$\Omega=1+\epsilon\lambda$ in the form \begin{equation}
A(\lambda,\epsilon)\equiv\left[\begin{array}{cc}
-\epsilon\sqrt{\lambda(\lambda-4)}I-(1+p)M(\lambda) & -pM(\lambda)\\
-pN(\kappa) & \sqrt{(2+\epsilon\lambda)^{2}+4\epsilon(2+\epsilon\lambda)}I-(1+p)N(\kappa)\end{array}\right],\label{coefficient-matrix-3}\end{equation}
 where $M(\lambda)\equiv Q(q_{1}^{+},q_{2}^{+},\cdots,q_{N-1}^{+})$,
$N(\kappa(\lambda))\equiv Q(q_{1}^{-},q_{2}^{-},\cdots,q_{N-1}^{-})$,
and the appropriate branches of $\sin z(\lambda)$ and $\sinh(\kappa(\lambda))$
are chosen in the domain $D_{\delta_{+}}$.

Let $|\lambda|={\cal O}(\epsilon^{-r})$ as $\epsilon\to0$ for $r\in[0,1)$.
Then, we have \begin{equation}
A(\lambda,\epsilon)\to\left[\begin{array}{cc}
-(1+p)M(\lambda) & -pM(\lambda)\\
-pI & (1-p)I\end{array}\right]\quad\mbox{{\rm as}}\quad\epsilon\to0,\label{limiting-matrix-1}\end{equation}
 where $M(\lambda)\to I$ as $\epsilon\to0$ if $r\in(0,1)$ and $M(\lambda)\nrightarrow I$
as $\epsilon\to0$ if $r=0$. The limiting matrix (\ref{limiting-matrix-1})
is not singular if $\gamma_{+}>4$. Hence $A(\lambda,\epsilon)$ is
not singular for small $\epsilon\geq0$ if $|\lambda|={\cal O}(\epsilon^{-r})$
with $r\in[0,1)$.

Let $|\lambda|={\cal O}(\epsilon^{-r})$ as $\epsilon\to0$ for $r\in(0,1]$.
Then, we have \begin{equation}
A(\lambda,\epsilon)\to\left[\begin{array}{cc}
-(1+\epsilon\lambda+p)I & -pI\\
-pI & (1+\epsilon\lambda-p)I\end{array}\right]\quad\mbox{{\rm as}}\quad\epsilon\to0.\label{limiting-matrix-2}\end{equation}
 Again, the limiting matrix is not singular if $\epsilon\lambda\neq-1$
(that is $\delta_{+}<1$) and hence $A(\lambda,\epsilon)$ is not
singular for small $\epsilon\geq0$ if $|\lambda|={\cal O}(\epsilon^{-r})$
with $r\in(0,1]$.

Since the above asymptotic scaling overlap at $r\in(0,1)$, the matrix
$A(\lambda,\epsilon)$ is not singular in the domain $D_{\delta_{+}}$
for small $\epsilon>0$. \end{proof}

Theorem \ref{theorem-resolvent} is proved with Lemma \ref{lemma-first-part},
Theorem \ref{theorem-continuous-spectrum}, and Lemma \ref{lemma-semi-annulus}.

\section{Perturbation arguments for the full resolvent}

Let us now consider the full spectral problem (\ref{spectral-stability}).
Thanks to Proposition \ref{proposition-existence} and expansion (\ref{power-series}),
we can represent $\phi_{n}^{2p}$ by \[
\phi_{n}^{2p}=\sum_{m\in U_{+}\cup U_{-}}\delta_{n,m}(1+\epsilon\chi_{m})+\epsilon^{2}W_{n},\]
 where $\{\chi_{m}\}_{m\in U_{+}\cup U_{-}}$ is a set of numerical
coefficients and $\{W_{n}\}_{n\in\Z}\in l^{2}(\Z)$ is a new potential
such that $\|{\bf W}\|_{l^{2}}={\cal O}(1)$ as $\epsilon\to0$.

In variables $\{(a_{n},b_{n})\}_{n\in\Z}$, the resolvent problem
can be rewritten in the operator form \begin{equation}
(\tilde{L}+\epsilon^{2}\tilde{W})\left[\begin{array}{cc}
{\bf a}\\
{\bf b}\end{array}\right]-\Omega\left[\begin{array}{cc}
{\bf a}\\
{\bf b}\end{array}\right]=\left[\begin{array}{cc}
{\bf f}\\
-{\bf g}\end{array}\right],\label{equation-full-resolvent}\end{equation}
 where \[
\tilde{L}=\left[\begin{array}{cc}
-\epsilon\Delta+I-(1+p)\tilde{V} & -p\tilde{V}\\
p\tilde{V} & \epsilon\Delta-I+(1+p)\tilde{V},\end{array}\right],\quad\tilde{W}=\left[\begin{array}{cc}
-(1+p)W & -pW\\
pW & (1+p)W,\end{array}\right],\]
 and $\tilde{V}$ is the associated compact potential such that \[
(\tilde{V}u)_{n}=\sum_{m\in U_{+}\cup U_{-}}\delta_{n,m}(1+\epsilon\chi_{m})u_{m},\quad n\in\Z.\]

Let us denote the solution of the inhomogeneous system (\ref{equation-full-resolvent})
by \begin{equation}
\left[\begin{array}{cc}
{\bf a}\\
{\bf b}\end{array}\right]=R(\Omega)\left[\begin{array}{cc}
{\bf f}\\
-{\bf g}\end{array}\right],\label{equation-full-resolvent-2}\end{equation}
 where $R(\Omega)$ is the resolvent operator of the full spectral
problem (\ref{spectral-stability}). The following theorem represents
the main result of our paper.

\begin{theorem} Fix $U_{+},U_{-}\subset\Z$ such that $U_{+}\cap U_{-}=\varnothing$,
$N:=|U_{+}|+|U_{-}|<\infty$, and $U_{+}\cup U_{-}$ is simply connected.
For any integer $p\geq2$, there are $\epsilon_{0}>0$ and $\delta>0$
such that for any fixed $\epsilon\in(0,\epsilon_{0})$ the resolvent
operator \[
R(\Omega):l^{2}(\Z)\times l^{2}(\Z)\to l^{2}(\Z)\times l^{2}(\Z)\]
 is bounded for any $\Omega\notin B_{\delta}(0)\cup[1,1+4\epsilon]\cup[-1-4\epsilon,-1]$.
Moreover, $R(\Omega)$ has exactly $2N$ poles (counting multiplicities)
inside $B_{\delta}(0)$ and admits the limits \[
R^{\pm}(\Omega):=\lim_{\mu\downarrow0}R(\Omega\pm i\mu)\]
 such that for any $\Omega\in[1,1+4\epsilon]\cup[-1-4\epsilon,-1]$
and any $\epsilon\in(0,\epsilon_{0})$, there is $C>0$ such that
\[
\|R^{\pm}(\Omega)\|_{l_{1}^{1}\times l_{1}^{1}\to l^{\infty}\times l^{\infty}}\leq C\epsilon^{-1}.\]
 \label{theorem-main} \end{theorem}

\begin{proof} Let $R_{\tilde{L}}(\Omega)$ be the resolvent operator
for the inverse operator $(\tilde{L}-\Omega I)^{-1}$ associated with
the compactly supported potential $\tilde{V}$. We shall prove that
Theorem \ref{theorem-resolvent} remains valid for the resolvent operator
$R_{\tilde{L}}(\Omega)$. Assuming it, the rest of the proof of Theorem
\ref{theorem-main} relies on the perturbation arguments and the resolvent
identities \[
R(\Omega)=R_{\tilde{L}}(\Omega)(I+\epsilon^{2}\tilde{W}R_{\tilde{L}}(\Omega))^{-1}=(I+\epsilon^{2}R_{\tilde{L}}(\Omega)\tilde{W})^{-1}R_{\tilde{L}}(\Omega).\]
 Indeed, outside the continuous spectrum located at \[
\sigma_{c}(\tilde{L}+\epsilon^{2}\tilde{W})=\sigma_{c}(\tilde{L})=\sigma_{c}(L)\equiv[-1-4\epsilon,-1]\cup[1,1+4\epsilon],\]
 the resolvent operator $R_{\tilde{L}}(\Omega)$ is only singular
inside the disk $B_{\delta_{0}}(0)$, where perturbation theory of
isolated eigenvalues apply. Inside the continuous spectrum, $R_{\tilde{L}}(\Omega)$
is extended as a bounded operator from $l_{1}^{1}(\Z)\times l^{1}(\Z)$
to $l^{\infty}(\Z)\times l^{\infty}(\Z)$ such that for any $\Omega\in[1,1+4\epsilon]$
and any $\epsilon\in(0,\epsilon_{0})$, there is $C>0$ such that
\begin{equation}
\exists C>0:\quad\|R_{\tilde{L}}^{\pm}(\Omega)\|_{l_{1}^{1}\times l_{1}^{1}\to l^{\infty}\times l^{\infty}}\leq C\epsilon^{-1}.\label{bound-above}\end{equation}
 Since $\tilde{W}$ is a bounded ($\Omega$,$\epsilon$)-independent
operator from $l^{\infty}(\Z)\times l^{\infty}(\Z)$ to $l_{1}^{1}(\Z)\times l_{1}^{1}(\Z)$
(note here that $\mbox{\boldmath\ensuremath{\phi}}\in l_{1/2}^{2}(\Z)$,
see Remark \ref{remark-existence}), bound (\ref{bound-above}) implies
that \[
\exists C>0:\quad\|\epsilon^{2}\tilde{W}R_{\tilde{L}}(\Omega)\|_{l_{1}^{1}\times l_{1}^{1}\to l_{1}^{1}\times l_{1}^{1}}\leq C\epsilon,\]
 so that $(I+\epsilon^{2}\tilde{W}R_{\tilde{L}}(\Omega))$ is an invertible
bounded operator from $l_{1}^{1}(\Z)\times l_{1}^{1}(\Z)$ to $l_{1}^{1}(\Z)\times l_{1}^{1}(\Z)$
for small $\epsilon>0$.

We only need to extend Theorem \ref{theorem-resolvent} to the resolvent
operator $R_{\tilde{L}}(\Omega)$. The Green's function representation
(\ref{resolvent-a-b}) and the linear system (\ref{system-a-b}) are
now written with the factor $(1+\epsilon\chi_{m})$ in the sum over
$m\in U_{+}\cup U_{-}$. This implies that the coefficient matrix
$A(\Omega,\epsilon)$ is now written as \[
\tilde{A}(\Omega,\epsilon):=\left[\begin{array}{cc}
2i\epsilon\sin z(\lambda_{+})I-(1+p)Q^{+}(\Omega,\epsilon)(I+\epsilon D) & -pQ^{+}(\Omega,\epsilon)(I+\epsilon D)\\
-pQ^{-}(\Omega,\epsilon)(I+\epsilon D) & 2i\epsilon\sin z(\lambda_{-})I-(1+p)Q^{-}(\Omega,\epsilon)(I+\epsilon D)\end{array}\right],\]
 where $D$ is a diagonal matrix of elements $\{\chi_{m}\}_{m\in U_{+}\cup U_{-}}$.
If $p\geq2$, Lemmas \ref{lemma-first-part}, \ref{lemma-A}, \ref{lemma-condition-bounded},
and \ref{lemma-semi-annulus} remain valid as these lemmas were proved
from the limit $\epsilon=0$ (perturbation theory of Appendix A is
only required for $p=1$), where $\tilde{A}(\Omega,0)=A(\Omega,0)$.
Therefore, Theorem \ref{theorem-resolvent} holds for the resolvent
operator $R_{\tilde{L}}(\Omega)$ if $p\geq2$. \end{proof}

\begin{corollary} The result of Theorem \ref{theorem-main} holds
for $p=1$ if $N=1$. \end{corollary}

\begin{proof} If $N=1$ (which is the case of fundamental discrete
soliton), the $2\times2$ coefficient matrix \[
\tilde{A}(\Omega,\epsilon)=\left[\begin{array}{cc}
2i\epsilon\sin z(\lambda_{+})-(1+p)(1+\epsilon\chi_{0}) & -p(1+\epsilon\chi_{0})\\
-p(1+\epsilon\chi_{0}) & 2i\epsilon\sin z(\lambda_{-})-(1+p)(1+\epsilon\chi_{0})\end{array}\right],\]
 is only singular in $B_{\delta}(0)$ for small $\epsilon>0$, where
a double pole of $R_{\tilde{L}}(\Omega)$ and $R(\Omega)$ resides.
\end{proof}

Unfortunately, in the cubic case $p=1$, we can not generally extend
the result of Theorem \ref{theorem-main} to multi-site discrete solitons
with $N\geq2$ because the perturbation theory for $\tilde{A}(\Omega,\epsilon)$
near the end points of the continuous spectrum $\Omega=\pm1$ and
$\Omega=\pm(1+4\epsilon)$ draws no conclusion in a general case.
For instance, reworking the perturbative arguments of Appendix A,
we obtain the necessary condition for ${\rm Null}(A_{\pm}(\epsilon))^{2}>{\rm Null}(A_{\pm}(\epsilon))$
in the form \[
\epsilon(2I-J-2D)w+{\cal O}(\epsilon^{2})\perp w\in{\rm Null}(M_{+}),\]
 where $I$ is the identity matrix in $\R^{N}$, $J$ is the two-diagonal
matrix (\ref{matrix-J}) from Appendix A, and $D$ is a diagonal matrix
of $\{\chi_{m}\}_{m\in U_{+}\cup U_{-}}$. Because $(2I-J-2D)$ is
no longer positive definite, the degenerate cases with ${\rm Null}(A_{\pm}(\epsilon))^{2}>{\rm Null}(A_{\pm}(\epsilon))$
are possible.

To illustrate this possibility, we set $N=3$ and consider three distinct
simply-connected discrete solitons associated with the sets \[
{\rm (a)}\; U_{+}=\{0,1,2\};\quad{\rm (b)}\; U_{+}=\{0,1\},\;\; U_{-}\{2\};\quad{\rm (c)}\; U_{+}=\{0,2\},\;\; U_{-}\{1\}.\]
 Computations of the power expansions (\ref{power-series}) give \[
{\rm (a)}\;\chi_{m}=\left\{ \begin{array}{c}
1,\quad m=0,\\
0,\quad m=1,\\
1,\quad m=2,\end{array}\right.\quad{\rm (b)}\;\chi_{m}=\left\{ \begin{array}{c}
1,\quad m=0,\\
2,\quad m=1,\\
3,\quad m=2,\end{array}\right.\quad{\rm (c)}\;\chi_{m}=\left\{ \begin{array}{c}
3,\quad m=0,\\
4,\quad m=1,\\
3,\quad m=2.\end{array}\right.\]
 As a result, matrix $C\equiv2I-J-2D$ is obtained in the form \[
{\rm (a)}\; C=\left[\begin{array}{ccc}
0 & -1 & 0\\
-1 & 2 & -1\\
0 & -1 & 0\end{array}\right],\quad{\rm (b)}\; C=\left[\begin{array}{ccc}
0 & -1 & 0\\
-1 & -2 & -1\\
0 & -1 & -4\end{array}\right],\quad{\rm (c)}\; C=\left[\begin{array}{ccc}
-4 & -1 & 0\\
-1 & -6 & -1\\
0 & -1 & -4\end{array}\right].\]
 We have \[
{\rm Null}(M_{+})={\rm span}\{w_{1},w_{2}\},\quad w_{1}=\frac{1}{\sqrt{2}}\left[\begin{array}{c}
1\\
0\\
-1\end{array}\right],\quad w_{2}=\frac{1}{\sqrt{6}}\left[\begin{array}{c}
1\\
-2\\
1\end{array}\right],\]
 from which we compute the matrix of projections $P_{ij}=\langle Cw_{i},w_{j}\rangle_{\C^{3}}$
in the form \[
{\rm (a)}\; P=\left[\begin{array}{cc}
0 & 0\\
0 & \frac{8}{3}\end{array}\right],\quad{\rm (b)}\; P=\left[\begin{array}{cc}
-2 & \frac{2}{\sqrt{3}}\\
\frac{2}{\sqrt{3}} & -\frac{2}{3}\end{array}\right],\quad{\rm (c)}\; P=\left[\begin{array}{cc}
-4 & 0\\
0 & -4\end{array}\right].\]
 The projection matrices in cases (a) and (b) are singular. In order
to show that ${\rm Null}(A_{\pm}(\epsilon))^{2}={\rm Null}(A_{\pm}(\epsilon))$
for $\epsilon\in(0,\epsilon_{0})$, we need to extend perturbation
arguments of Appendix A to the order ${\cal O}(\epsilon^{2})$. Although
it is quite possible that the non-degeneracy condition ${\rm Null}(A_{\pm}(\epsilon))^{2}={\rm Null}(A_{\pm}(\epsilon))$
is still satisfied for simply-connected multi-site discrete solitons
for $p=1$, we do not include computations of the higher-order perturbation
theory in this paper.

\section{Case study for a non-simply-connected two-site soliton}

We explain now why the resolvent operator associated with non-simply-connected
multi-site discrete solitons have singularities near the anti-continuum
limit. These singularities appear in Lemma \ref{lemma-A} because
the determinant $D_{N}(q_{1},q_{2},\cdots,q_{N-1})$ given by \ref{determinant-D-N}
has zeros for $\theta\in(-\pi,0)$.

Let us consider a case study of a two-site soliton with $n_{1}=0$
and $n_{2}=m\geq2$. For clarity of presentation, we only consider
$p\geq2$. The power series expansions (\ref{power-series}) give
\begin{equation}
m\geq3:\quad\phi_{n}^{2p}=(\delta_{n,0}+\delta_{n,m})\left(1+2\epsilon-2\epsilon^{2}\right)+\epsilon^{3}W_{n},\quad n\in\Z,\label{m-greater-2}\end{equation}
 and \begin{equation}
m=2:\quad\phi_{n}^{2p}=(\delta_{n,0}+\delta_{n,m})\left(1+2\epsilon-3\epsilon^{2}\right)+\epsilon^{3}W_{n},\quad n\in\Z,\label{m-equal-2}\end{equation}
 where $\{W_{n}\}_{n\in\Z}\in l^{2}(\Z)$ is a new potential such
that $\|{\bf W}\|_{l^{2}}={\cal O}(1)$ as $\epsilon\to0$.

Let us consider the coefficient matrix $A(\theta,\epsilon)$ at the
continuous spectrum $[1,1+4\epsilon]$ defined by (\ref{coefficient-matrix-2}).
We have explicitly \[
M(\theta)=\left[\begin{array}{cc}
1 & e^{-im\theta}\\
e^{-im\theta} & 1\end{array}\right],\quad N(\kappa)=\left[\begin{array}{cc}
1 & e^{-2\kappa}\\
e^{-2\kappa} & 1\end{array}\right].\]
 Note that ${\rm det}M(\theta)=1-e^{-2im\theta}$. Besides the end
points $\theta=-\pi$ and $\theta=0$, the matrix $M(\theta)$ (and,
therefore, the limiting matrix $A(\theta,0)$) is singular at the
intermediate points $\theta_{j}=-\frac{\pi j}{m}$ for $j=1,2,...,m-1$.

If $m=2$, there is only one intermediate-point singularity of $A(\theta,0)$
at $\theta=-\frac{\pi}{2}$. We have ${\rm dim}{\rm Null}A(-\frac{\pi}{2},0)=1$
and \[
{\rm Null}A^{*}\left(-\frac{\pi}{2},0\right)={\rm span}\left\{ e_{1}\right\} ,\quad e_{1}=\left[\begin{array}{c}
1\\
1\\
0\\
0\end{array}\right].\]
 The first two entries of the right-hand-side vector $h(\theta,\epsilon)$
in the linear system (\ref{abstract-linear-system}) are given explicitly
by \[
h_{1}(\theta,\epsilon)=\sum_{n\in\Z}e^{-i\theta|n|}f_{n},\quad h_{2}(\theta,\epsilon)=\sum_{n\in\Z}e^{-i\theta|n-2|}f_{n}.\]
 The constraint $\langle e_{1},h(-\frac{\pi}{2},0)\rangle_{\C^{4}}=0$
of Lemma \ref{lemma-fredholm} gives $h_{1}(-\frac{\pi}{2},0)=-h_{2}(-\frac{\pi}{2},0)$
and it is equivalent to the constraint $f_{1}=0$. If $f\in l^{1}(\Z)$
with $f_{1}\neq0$, then the solution of the linear system (\ref{system-a-b})
and hence the resolvent operator (\ref{resolvent-a-b}) has a singularity
at $\Omega=1+2\epsilon$ ($\theta=-\frac{\pi}{2}$) as $\epsilon\to0$.
This singularity indicates a resonance at the mid-point of the continuous
spectrum in the anti-continuum limit.

We would like to show that the resonance does not actually occur at
the continuous spectrum if $\epsilon>0$ and does not lead to (unstable)
eigenvalues off the continuous spectrum. To do so, we use the perturbation
theory up to the quadratic order in $\epsilon$.

Expanding solutions of the transcendental equation \[
2\epsilon(\cosh(\kappa)-1)=2+\epsilon\omega,\quad\omega=2-2\cos(\theta),\]
 we obtain \[
e^{-\kappa}=\frac{1}{2}\epsilon-\frac{2+\omega}{4}\epsilon^{2}+{\cal O}(\epsilon^{3})\quad\mbox{{\rm as}}\quad\epsilon\to0\]
 and \[
2\epsilon{\rm sinh}(\kappa)=2+(2+\omega)\epsilon-\epsilon^{2}+{\cal O}(\epsilon^{3})\quad\mbox{{\rm as}}\quad\epsilon\to0.\]
 Using expansion (\ref{m-equal-2}) for $m=2$, we obtain the extended
coefficient matrix $\tilde{A}(\theta,\epsilon)$ in the form \[
\tilde{A}(\theta,\epsilon):=\left[\begin{array}{cc}
2i\epsilon\sin(\theta)I-(1+p)\nu(\epsilon)M(\theta) & -p\nu(\epsilon)M(\theta)\\
-p\nu(\epsilon)N(\kappa) & 2\epsilon\sinh(\kappa)I-(1+p)\nu(\epsilon)N(\kappa)\end{array}\right],\]
 where $\nu(\epsilon)=1+2\epsilon-3\epsilon^{2}+{\cal O}(\epsilon^{3})$.
Using MATHEMATICA, we expand roots of $\det\tilde{A}(\theta,\epsilon)=0$
near $\theta=-\frac{\pi}{2}$ and $\epsilon=0$ to obtain \begin{equation}
\theta=-\frac{\pi}{2}+(p-1)\epsilon+2(1-p)\epsilon^{2}+i(p-1)^{2}\epsilon^{2}+{\cal O}(\epsilon^{3})\quad\mbox{{\rm as}}\quad\epsilon\to0.\label{interm_singularities}\end{equation}
 Since ${\rm Im}(\theta)>0$ for small $\epsilon>0$ and $z(\lambda_{+})=\theta$,
the solution of the linear system (\ref{abstract-linear-system})
is singular at the point $z(\lambda_{+})$, which does not belong
to the domain ${\rm Im}z(\lambda_{+})<0$ and hence violates the condition
(\ref{transcendental-equation-1}).

The singularity of the solution of the linear system (\ref{abstract-linear-system})
is still located near the continuous spectrum for small $\epsilon>0$
and, therefore, the resolvent operator $R(\Omega)$ becomes large
near the points $\Omega=\pm(1+2\epsilon)$ (although, it is always
a bounded operator from $l_{\sigma}^{2}(\Z)\times l_{\sigma}^{2}(\Z)$
to $l_{-\sigma}^{2}(\Z)\times l_{-\sigma}^{2}(\Z)$ for small $\epsilon>0$
and fixed $\sigma>\frac{1}{2}$). Since $\sin(\theta)$ is nonzero
for $\theta=-\frac{\pi}{2}$, the norm of $R(\Omega)$ is proportional
to the $2$-norm of inverse matrix $\tilde{A}^{-1}(\theta,\epsilon)$.

Figure \ref{fig:coef_mtrx_levels} illustrates the singularities of
the resolvent operator $R(\Omega)$ by plotting pseudospectra of the
coefficient matrix $A(\Omega,\epsilon)$ in the complex $\Omega$-plane
for $p=2$ and $\epsilon=0.05$. The subplots (a) and (b) for $m=1$
show that the matrix is singular at the edges of the continuous spectrum
$\Omega=\pm1$ and $\Omega=\pm(1+2\epsilon)$, and at four points
on the imaginary axis, the latter being attributed to the splitting
of zero eigenvalue in the anti-continuum limit. The subplots (c) and
(d) for $m=2$ and $m=3$ respectively show that in addition to singularities
at the edges of continuous spectrum there are also $m-1$ local maxima
at its intermediate points. This local maxima correspond to the minima
of ${\rm det}A(\Omega,\epsilon)$. We also notice the wedges on the
level sets as they cross the continuous spectrum occuring due to the
jump discontinuities in $z(\lambda_{+})$ and $A(\Omega,\epsilon)$
because the resolvent operator $R(\Omega)$ is discontinuous across
the continuous spectrum.

Figure \ref{fig:cont_spec} further illustrates what exactly happens
at the continuous spectrum. On the left, we plot $\left\Vert A(\Omega,\epsilon)^{-1}\right\Vert _{2}$
versus $\theta\in(-\pi,0)$ for the case $m=2$. On the right, we
show that the height of the local maxima near $\theta=-\pi/2$ is
proportional to $\epsilon^{-2}$ as prescribed by formula (\ref{interm_singularities}).

Figure \ref{fig:resolvent_levels} gives an illustration for pseudospectra
of the resolvent operator $R(\Omega)$. Recall that on the continuous
spectrum $\Omega\in[1,1+4\epsilon]$, $R(\Omega)$ is a bounded operator
from $l_{\sigma}^{2}(\Z)\times l_{\sigma}^{2}(\Z)$ to $l_{-\sigma}^{2}(\Z)\times l_{-\sigma}^{2}(\Z)$)
for fixed $\sigma>\frac{1}{2}$. To incorporate the weighted $l^{2}$
spaces, we consider the renormalized resolvent operator \[
\tilde{R}_{L}(\Omega)=(\tilde{L}-\Omega\tilde{I}_{2})^{-1}:l^{2}(\Z)\times l^{2}(\Z)\to l^{2}(\Z)\times l^{2}(\Z),\]
 where $\tilde{L}$ is derived from $L$ by replacing operators $I$,
$\Delta$ and $V$ with $\tilde{I}$, $\tilde{\Delta}$ and $\tilde{V}$,
and $\tilde{I}_{2}=\mathrm{diag}\{\tilde{I},\tilde{I}\}$. Here \[
\tilde{I}_{n,m}=\kappa_{n}^{2}\delta_{n,m},\quad\tilde{V}_{n,m}=\tilde{I}_{n,m}\sum_{j\in U_{+}\cup U_{-}}\delta_{n,j},\]
 \[
\tilde{\Delta}_{n,n}=-2\kappa_{n}^{2},\quad\tilde{\Delta}_{n,n+1}=\tilde{\Delta}_{n+1,n}=\kappa_{n}\kappa_{n+1},\]
 and $\kappa_{n}=(1+n^{2})^{\sigma/2}$. The lattice problem is considered
for $2K+1$ grid points and the corresponding matrix representation
of operators $\tilde{L}$ and $\tilde{I}_{2}$ is constructed subject
to the Dirichlet boundary conditions.

The level sets for the $(2K+1)\times(2K+1)$ matrix approximation
of the resolvent $\tilde{R}(\Omega)$ are plotted on Figure \ref{fig:resolvent_levels}.
The subplots of Figure \ref{fig:resolvent_levels} correspond to the
subplots of Figure \ref{fig:coef_mtrx_levels}. We observe that the
norm of $\tilde{R}(\Omega)$ has the same global behaviour as for
the norm of $A(\Omega,\epsilon)^{-1}$. However, the resolvent operator
$\tilde{R}(\Omega)$ has no singularities at the edges $\Omega=\pm1$
and $\Omega=\pm(1+4\epsilon)$ because these singularities are canceled
according to Lemma \ref{lemma-condition-bounded} (which remains true
for any $m\geq1$, see Remark \ref{remark-lemma-4}).

Although no arguments exist to exclude resonances at the mid-point
of the continuous spectrum for the linearized dNLS equation (\ref{spectral-stability}),
the case study of a two-site discrete soliton suggests that the resonances
do not happen at the continuous spectrum for small but finite values
of $\epsilon>0$. Moreover, the resonances do not bifurcate to the
isolated eigenvalues off the continuous spectrum because isolated
eigenvalues near the continuous spectrum would violate the count of
unstable eigenvalues (\ref{count-eigenvalues-explicit}). Therefore,
the only scenario for these resonances is to move to the resonant
poles on the wrong sheets ${\rm Im}(z(\lambda_{\pm}))>0$ of the definition
of $z(\lambda_{\pm})$.

\begin{figure}
\begin{centering}
\begin{tabular}{cc}
\includegraphics[width=0.47\textwidth]{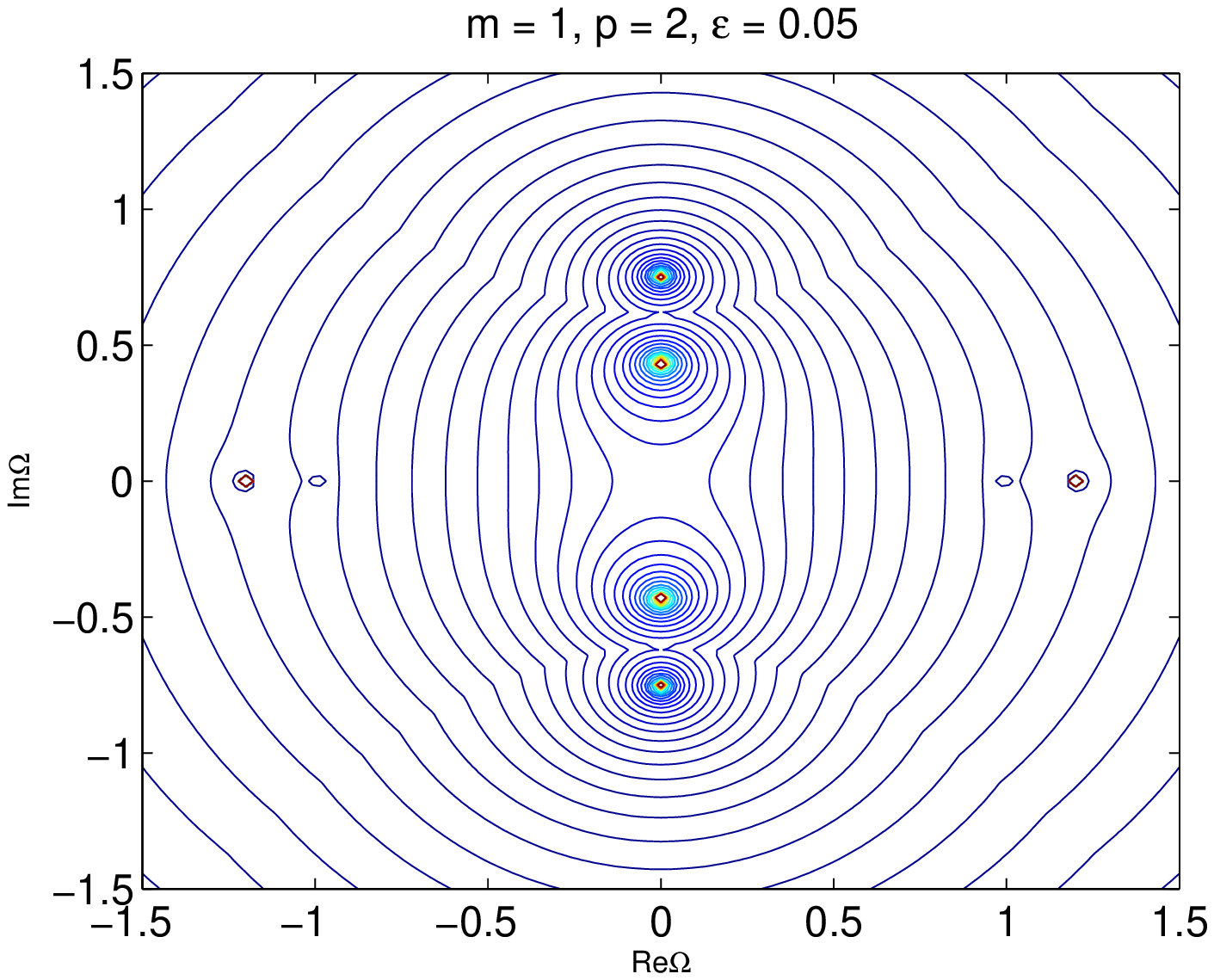}  & \includegraphics[width=0.47\textwidth]{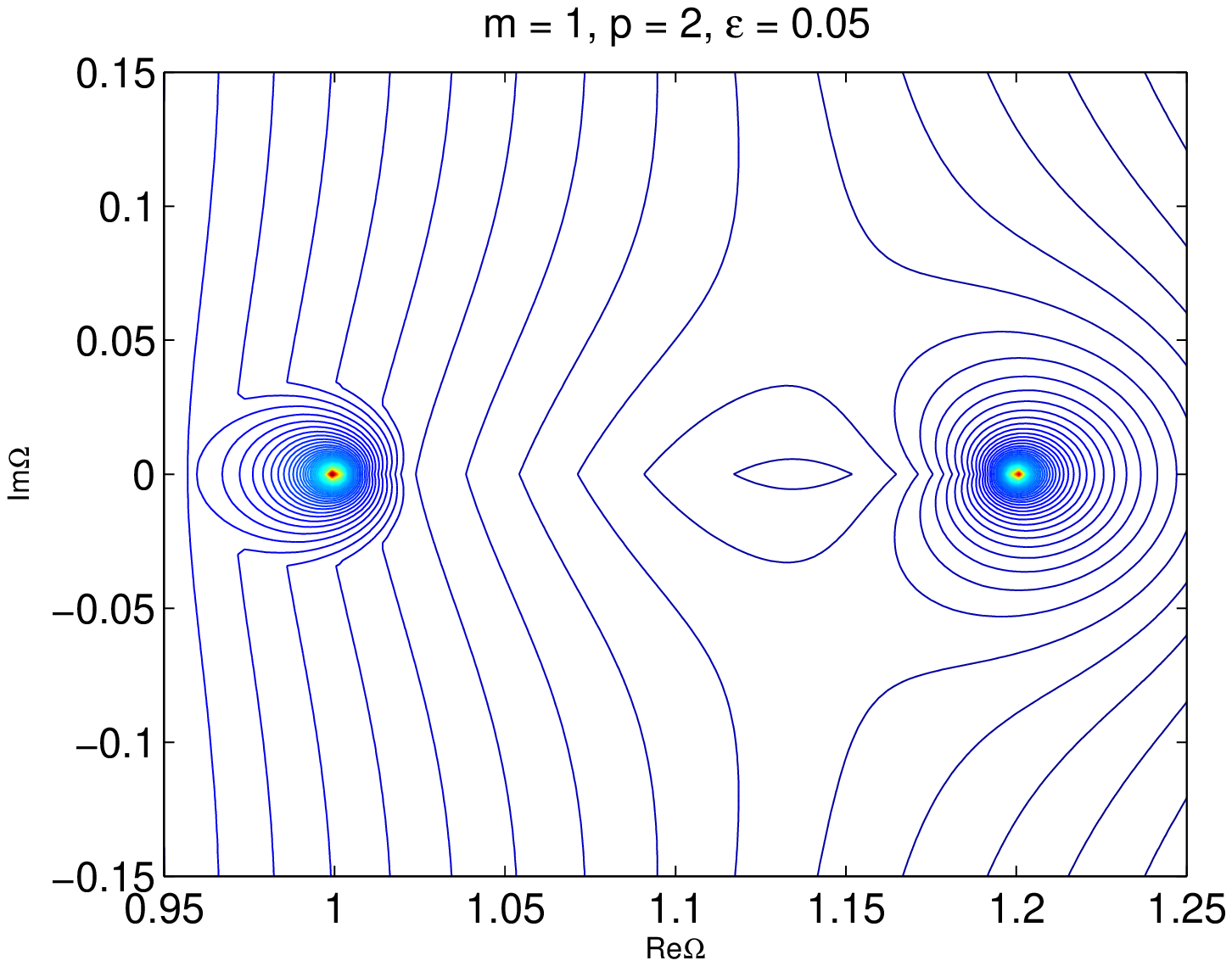}\tabularnewline
(a)  & (b)\tabularnewline
\includegraphics[width=0.47\textwidth]{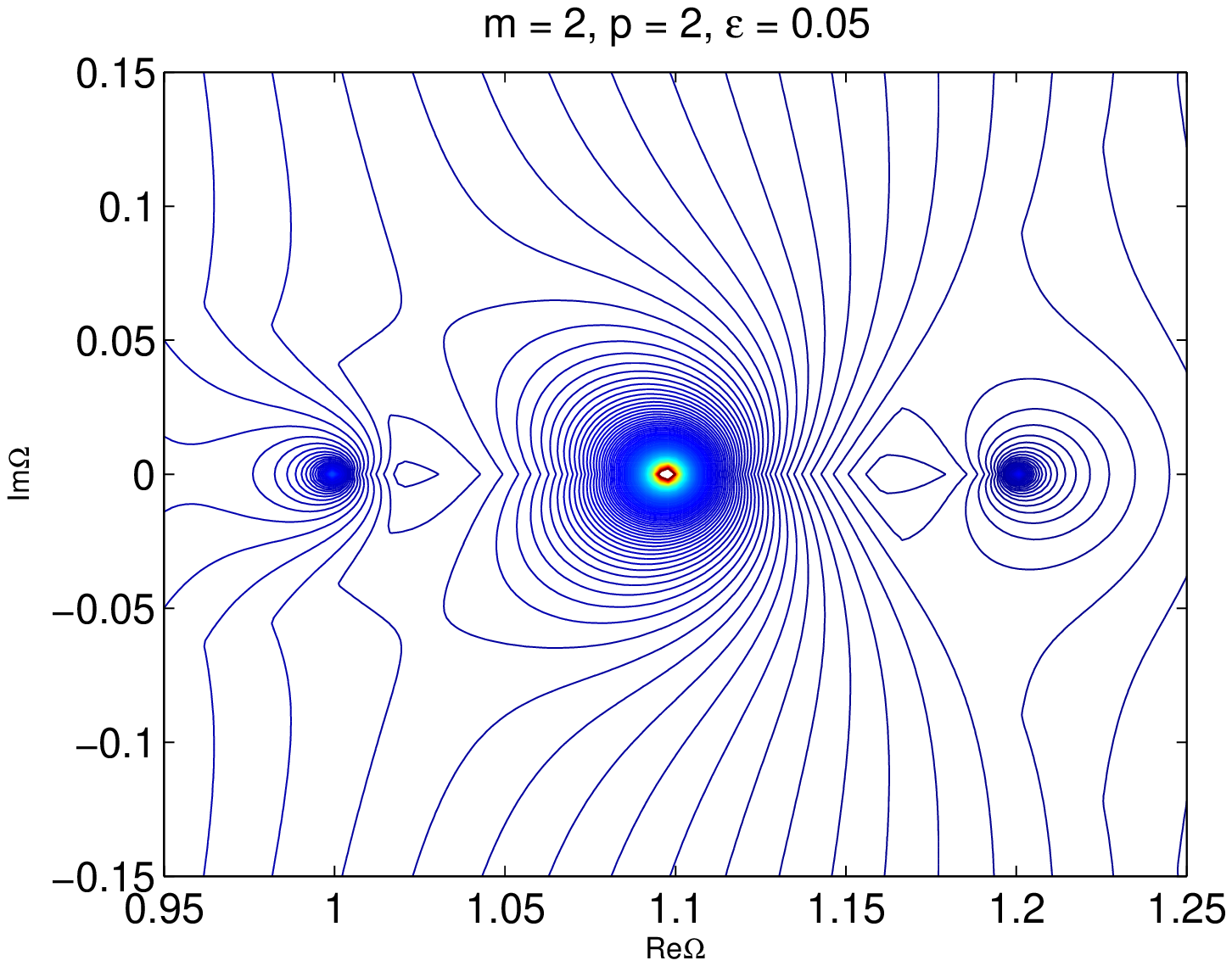}  & \includegraphics[width=0.47\textwidth]{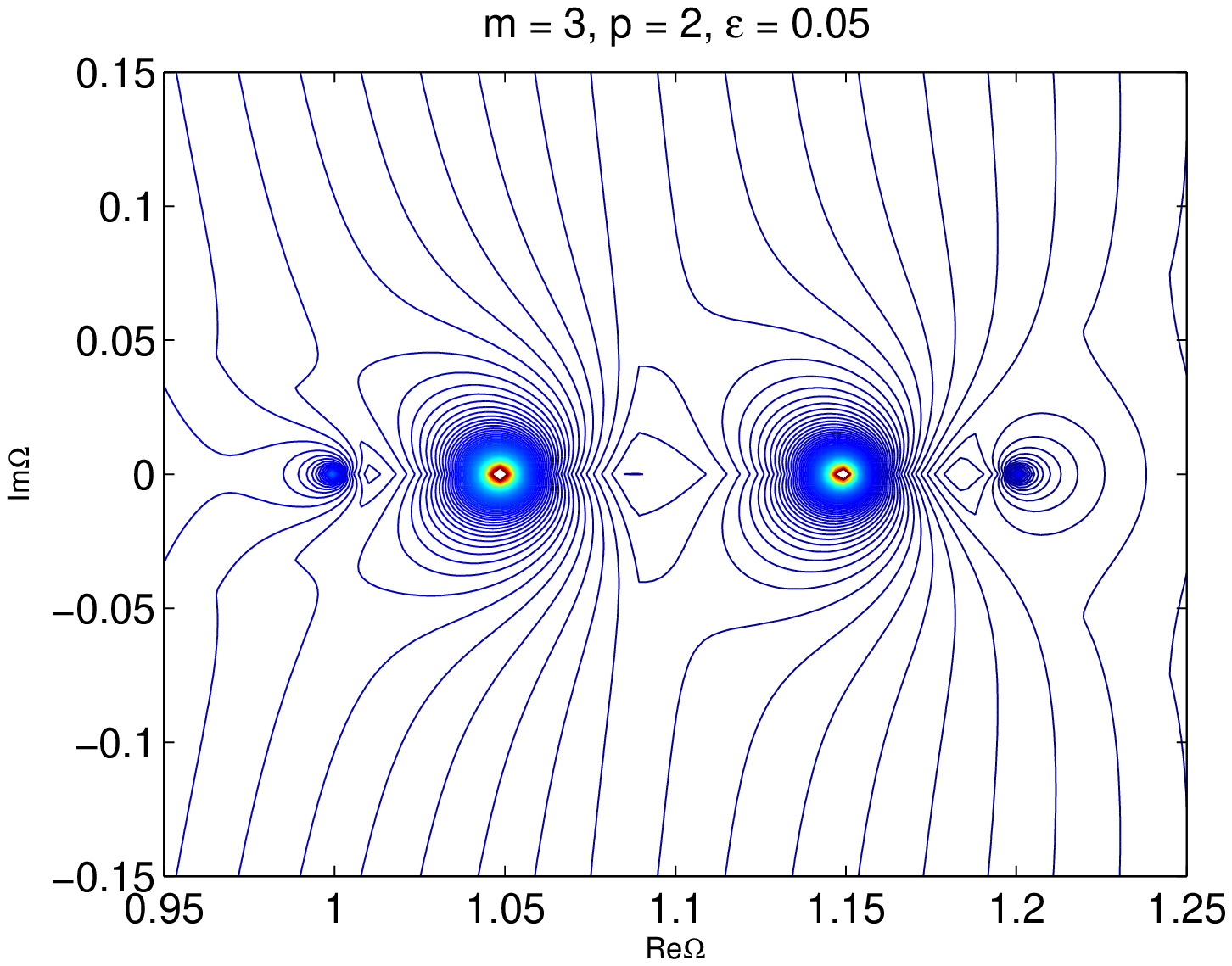}\tabularnewline
(c)  & (d)\tabularnewline
\end{tabular}
\par\end{centering}

\caption{Level sets for $\left\Vert A(\Omega,\epsilon)^{-1}\right\Vert _{2}$
in the $\Omega$-plane. The levels are equidistant on a logarithmic
scale. \label{fig:coef_mtrx_levels}}

\end{figure}

\begin{figure}
\begin{centering}
\begin{tabular}{cc}
\includegraphics[width=0.47\textwidth]{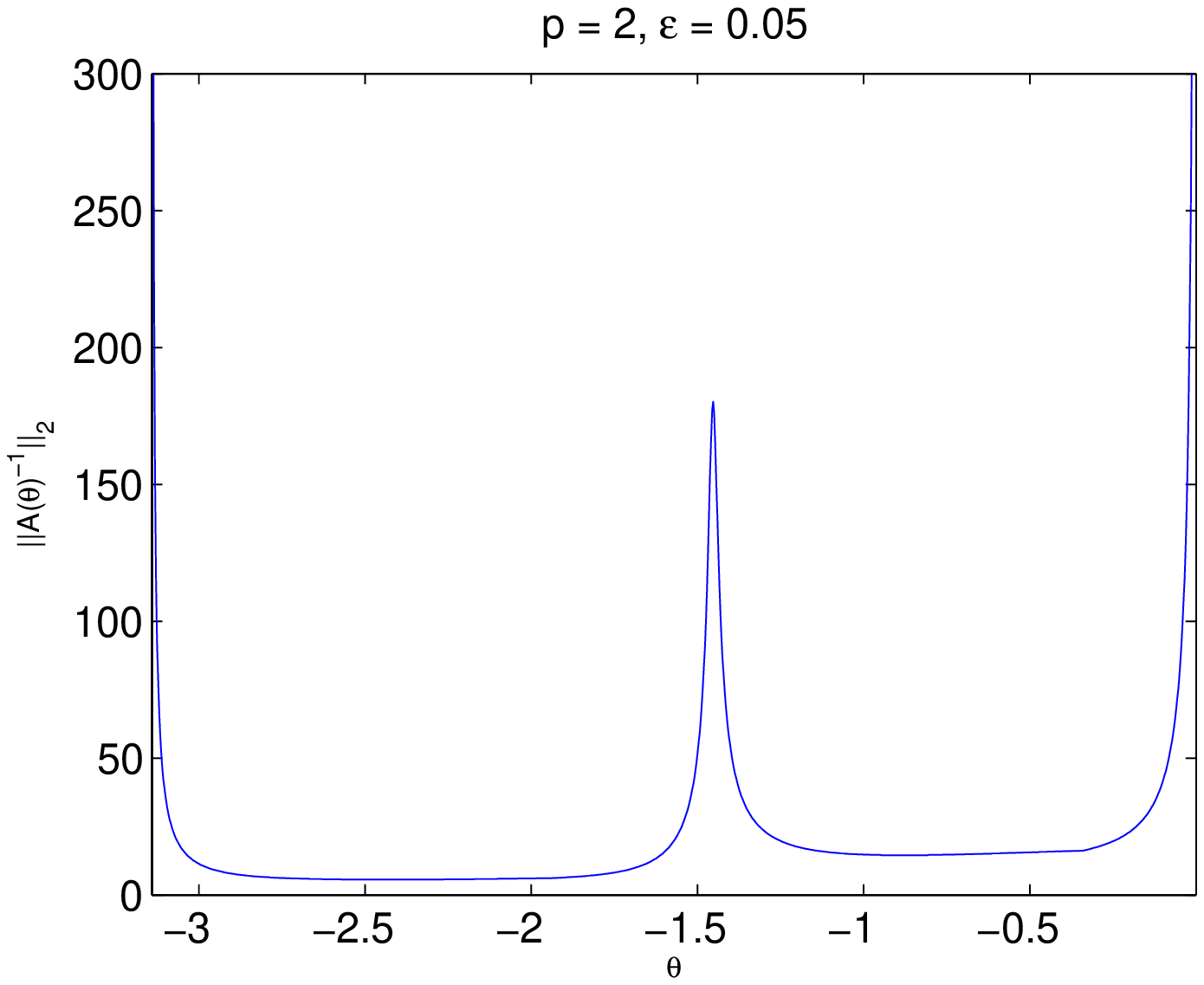}  & \includegraphics[width=0.47\textwidth]{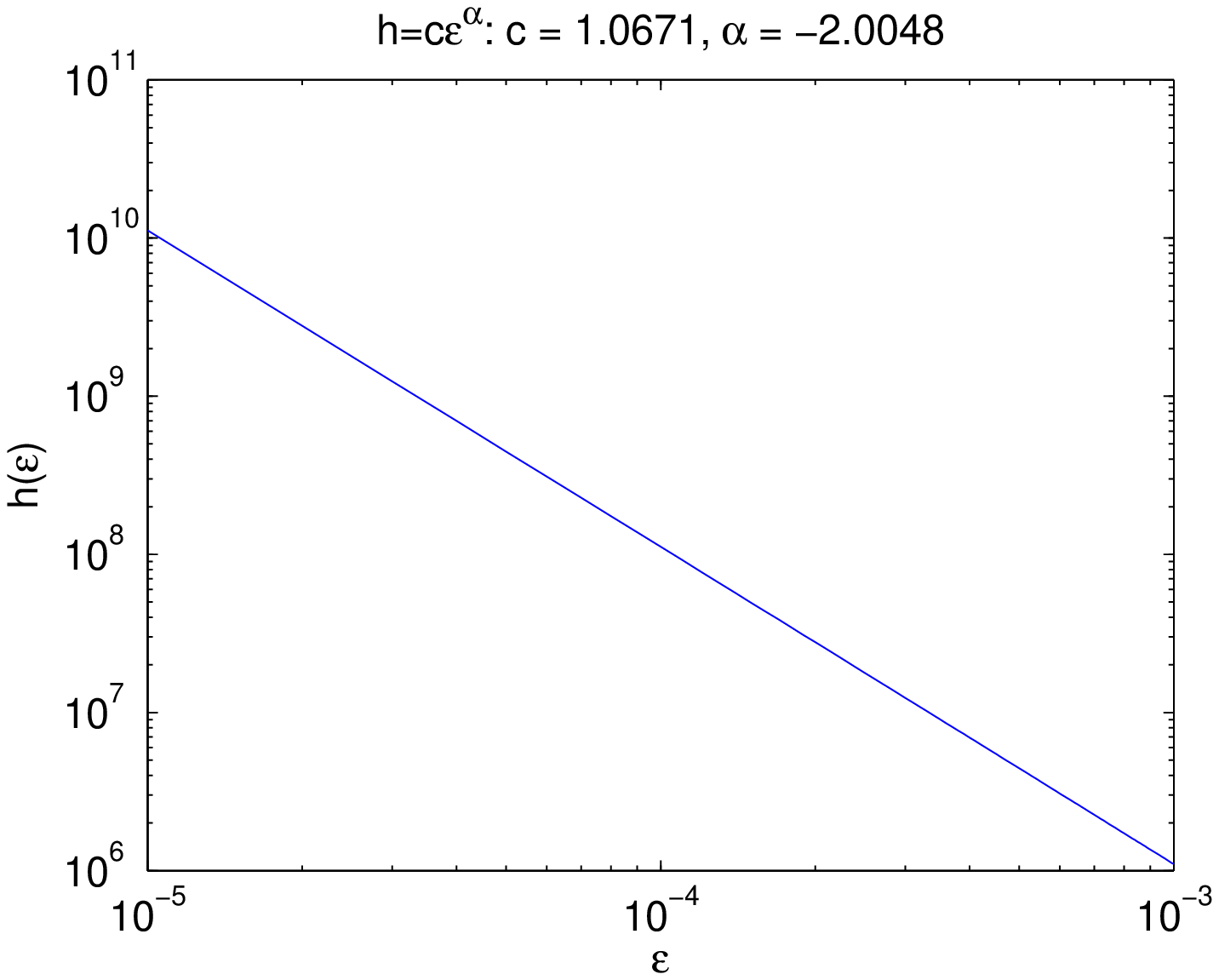}\tabularnewline
\end{tabular}
\par\end{centering}

\caption{Left: Norm $\left\Vert A(\Omega,\epsilon)^{-1}\right\Vert _{2}$ versus
$\theta\in(-\pi,0)$ for $m=2$. Right: The value of local maxima
of $\left\Vert A(\Omega,\epsilon)^{-1}\right\Vert _{2}$ in the neighborhood
of $\theta=-\pi/2$ as a function of $\epsilon$. \label{fig:cont_spec}}

\end{figure}

\begin{figure}
\begin{centering}
\begin{tabular}{cc}
\includegraphics[width=0.47\textwidth]{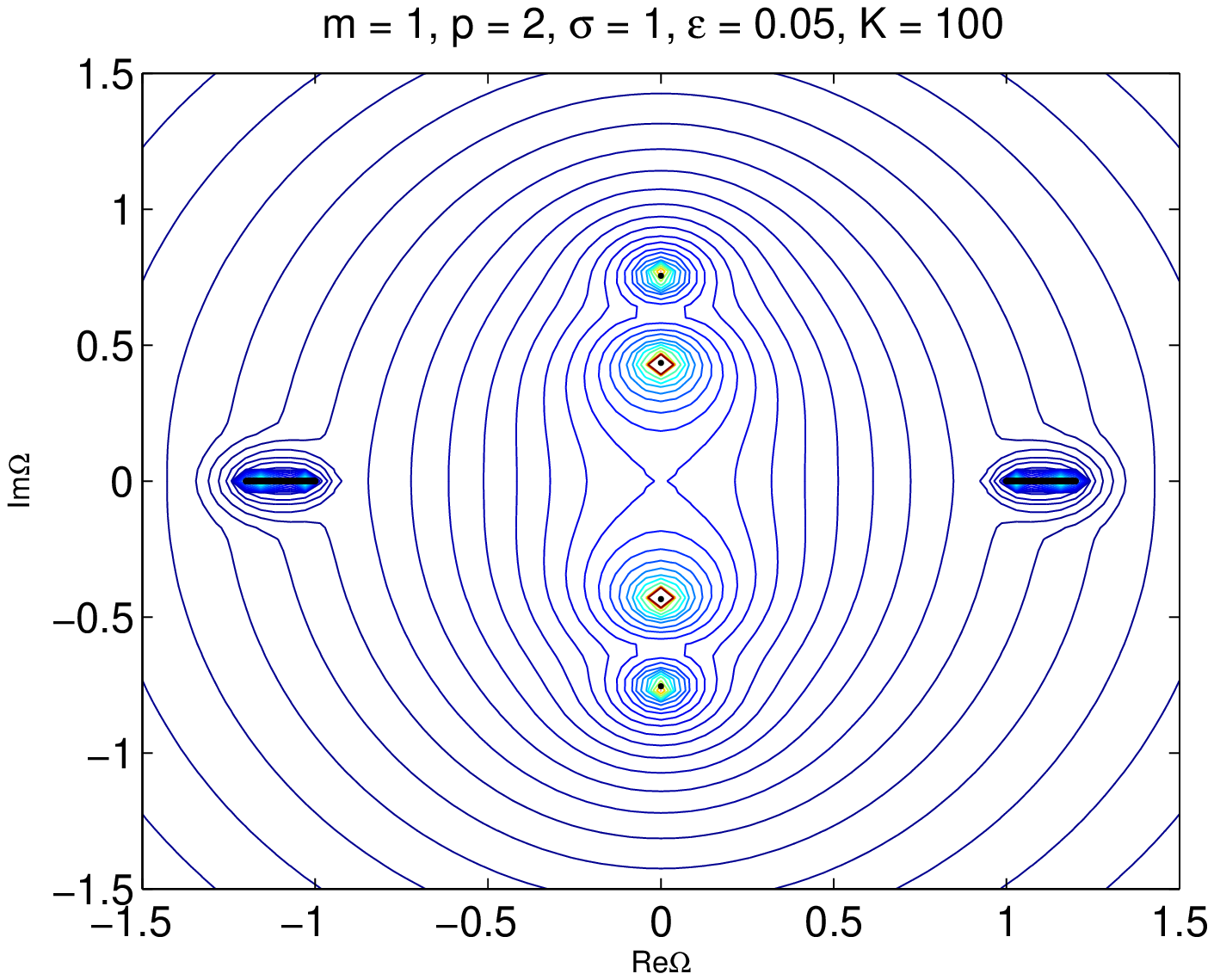}  & \includegraphics[width=0.47\textwidth]{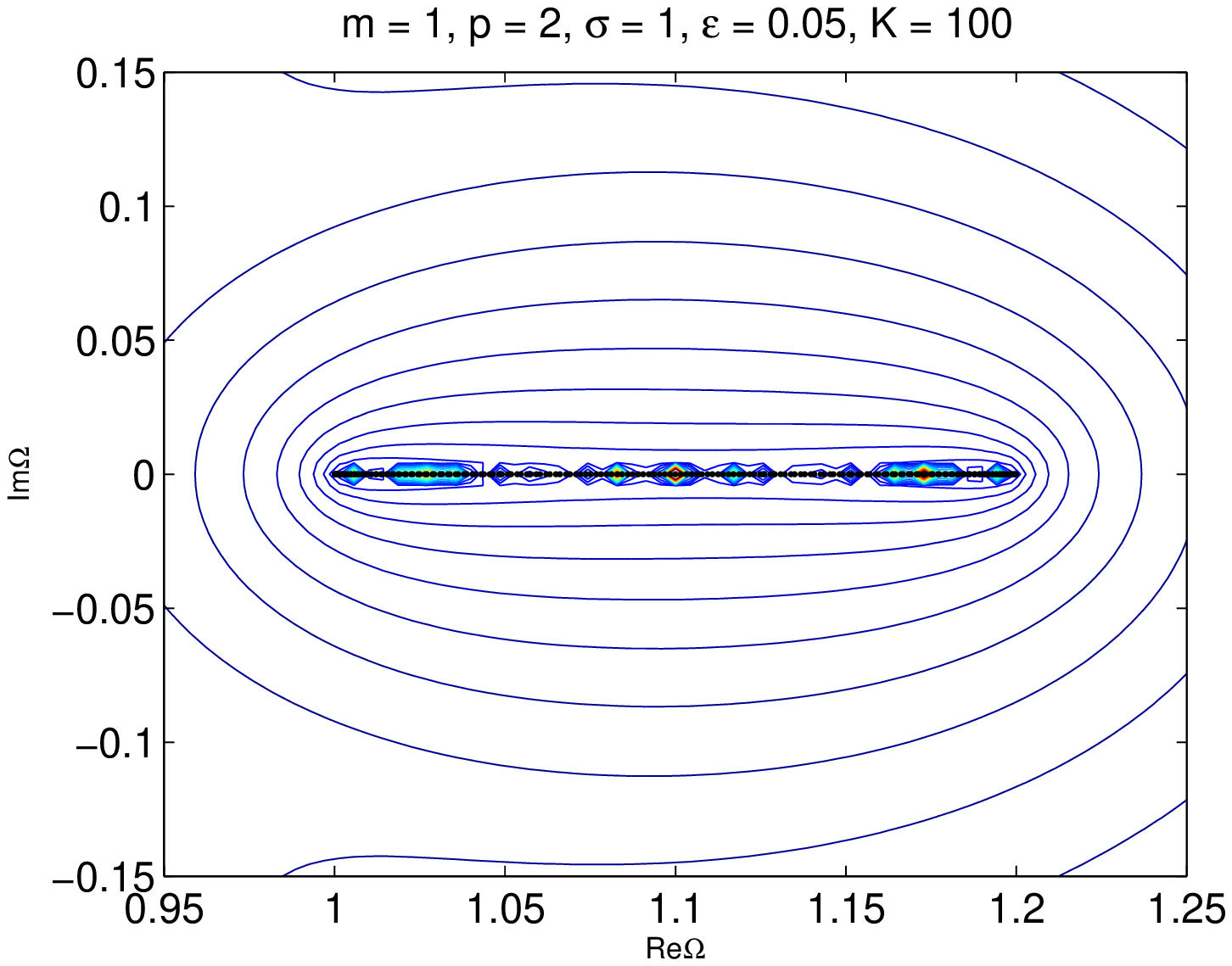}\tabularnewline
(a)  & (b)\tabularnewline
\includegraphics[width=0.47\textwidth]{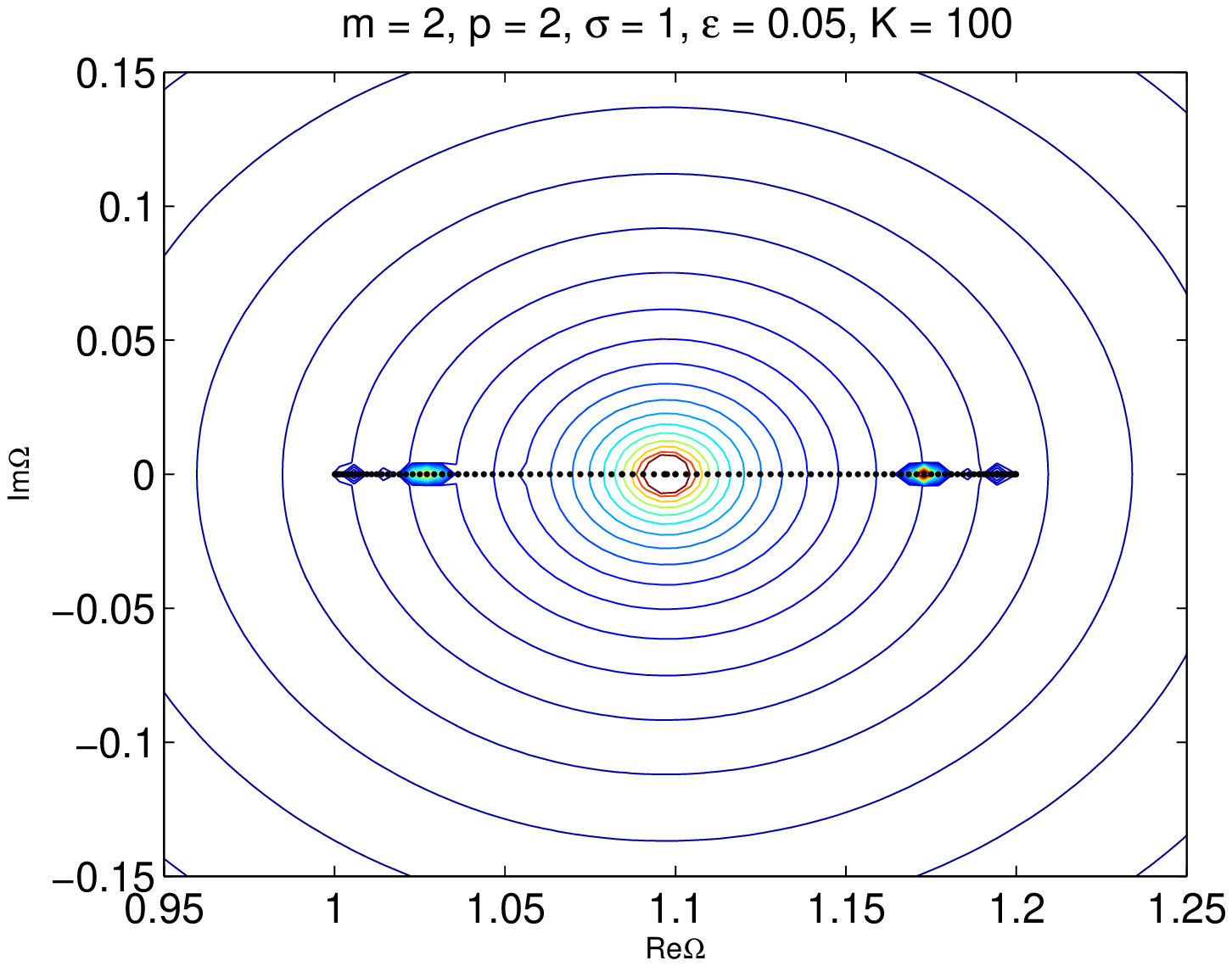}  & \includegraphics[width=0.47\textwidth]{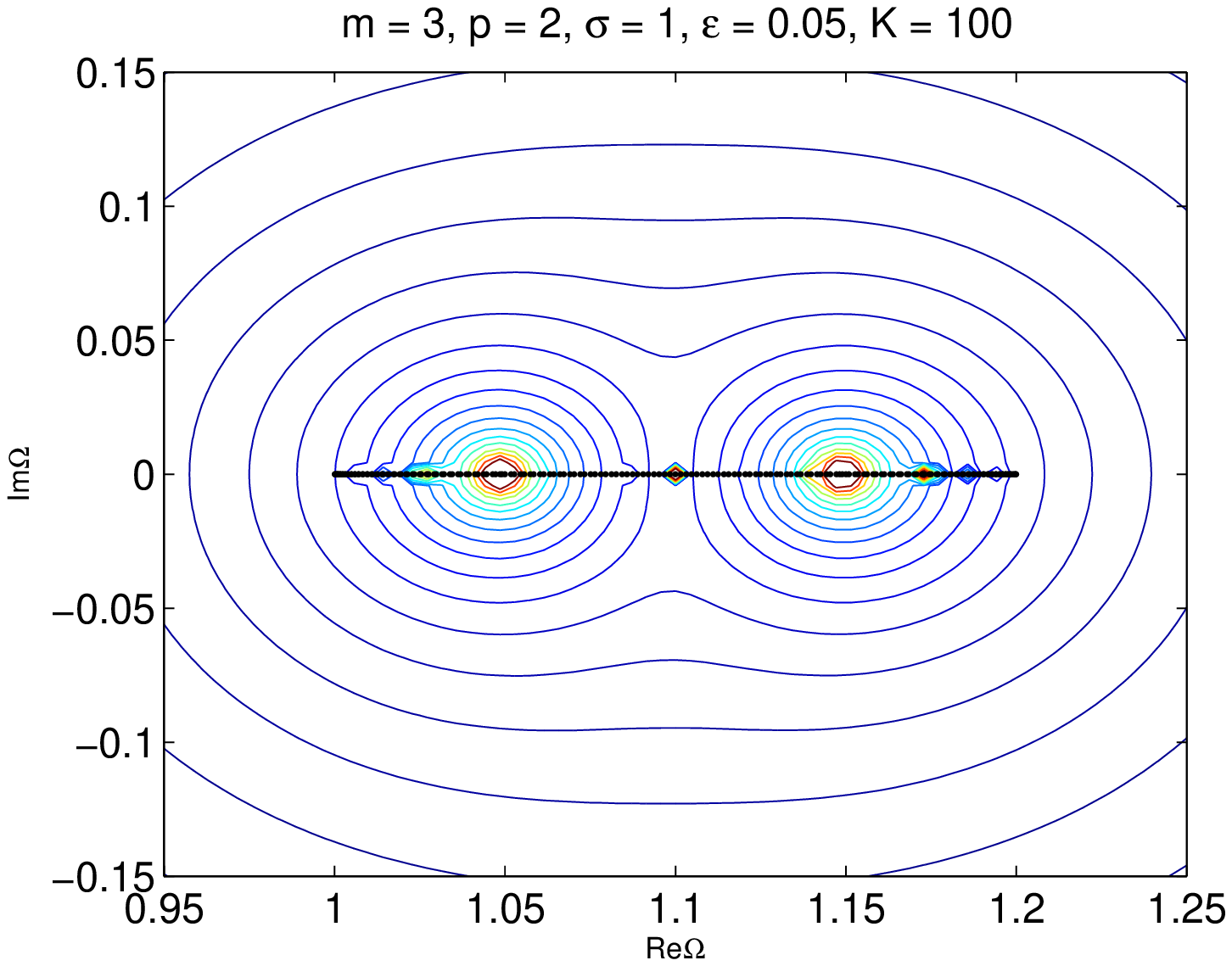}\tabularnewline
(c)  & (d)\tabularnewline
\end{tabular}
\par\end{centering}

\caption{The level sets of $\left\Vert (\tilde{L}-\Omega\tilde{I}_{2})^{-1}\right\Vert _{2}$
in the $\Omega$-plane. The black dots represent eigenvalues of the
matrix representation of operator $\tilde{L}$. The levels are equidistant
on a logarithmic scale. \label{fig:resolvent_levels}}

\end{figure}

\appendix

\section{Perturbative arguments for the cubic dNLS equation}

We recall the coefficient matrices $A_{\pm}(\epsilon)$ from the proof
of Lemma \ref{lemma-A}. In the case $p=1$ (the cubic dNLS equation),
these matrices are rewritten in the form \[
A_{\pm}(\epsilon)=\left[\begin{array}{cc}
-2M_{\pm} & -M_{\pm}\\
-N(\kappa_{\pm}) & 2\epsilon\sinh(\kappa_{\pm})I-2N(\kappa_{\pm})\end{array}\right],\]
 where $\kappa_{\pm}>0$ are uniquely defined by \[
2\epsilon(\cosh(\kappa_{+})-1)=2,\quad2\epsilon(\cosh(\kappa_{-})-1)=2+4\epsilon.\]

We recall that ${\rm Null}(A_{\pm}(\epsilon))$ and ${\rm Null}(M_{\pm})$
are $(N-1)$-dimensional for any $\epsilon\in[0,\epsilon_{0})$. It
is clear from the explicit form of $A_{\pm}^{*}(\epsilon)$ that \[
u\in{\rm Null}(A_{\pm}^{*}(\epsilon))\quad\Leftrightarrow\quad u=\left[\begin{array}{c}
w\\
0\end{array}\right],\quad w\in{\rm Null}(M_{\pm}).\]

At $\epsilon=0$, we also recall that ${\rm Null}(A_{\pm}(0))^{2}$
is $(2N-2)$-dimensional because of $(N-1)$ eigenvectors and $(N-1)$
generalized eigenvectors, \[
A_{\pm}(0)\left[\begin{array}{c}
0\\
w\end{array}\right]=\left[\begin{array}{c}
0\\
0\end{array}\right],\quad A_{\pm}(0)\left[\begin{array}{c}
-w\\
0\end{array}\right]=\left[\begin{array}{c}
0\\
w\end{array}\right],\quad w\in{\rm Null}(M_{\pm}).\]

We would like to show that ${\rm Null}(A_{\pm}(\epsilon))^{2}={\rm Null}(A_{\pm}(\epsilon))$
is $(N-1)$-dimensional for any $\epsilon\in(0,\epsilon_{0})$. In
other words, we would like to show that no solution $\tilde{u}\in\C^{2N}$
of the inhomogeneous equation $A_{\pm}(\epsilon)\tilde{u}=u\in{\rm Null}(A_{\pm}(\epsilon))$
exists for $\epsilon\in(0,\epsilon_{0})$. This task is achieved by
the perturbation theory. We will only consider the case $A_{+}(\epsilon)$,
which corresponds to $\theta=0$. The case $A_{-}(\epsilon)$ which
corresponds to $\theta=-\pi$ can be considered similarly.

We shall only consider the case of the simply-connected set $U_{+}\cup U_{-}$
with $m_{1}=m_{2}=...=m_{N-1}=1$. The general case holds without
any changes.

Thanks to the asymptotic expansions \[
e^{-\kappa_{+}}=\frac{\epsilon}{2}+{\cal O}(\epsilon^{2}),\quad2\epsilon\sinh(\kappa_{+})=2+2\epsilon+{\cal O}(\epsilon^{2}),\quad\mbox{{\rm as}}\quad\epsilon\to0,\]
 we obtain the asymptotic expansion \[
A_{+}(\epsilon)=\left[\begin{array}{cc}
-2M_{+} & -M_{+}\\
-I & O\end{array}\right]+\epsilon\left[\begin{array}{cc}
O & O\\
-\frac{1}{2}J & 2I-J\end{array}\right]+{\cal O}(\epsilon^{2}),\]
 where $I$ and $O$ are identity and zero matrices in $\R^{N}$ and
$J$ is the three-diagonal matrix in $\R^{N}$ \begin{equation}
J=\left[\begin{array}{cccccc}
0 & 1 & 0 & \cdots & 0 & 0\\
1 & 0 & 1 & \cdots & 0 & 0\\
0 & 1 & 0 & \cdots & 0 & 0\\
\vdots & \vdots & \vdots & \vdots & \vdots & \vdots\\
0 & 0 & 0 & \cdots & 0 & 1\\
0 & 0 & 0 & \cdots & 1 & 0\end{array}\right].\label{matrix-J}\end{equation}
 Note that $(2I-J)$ is a strictly positive matrix because it appears
in the finite-difference approximation of the differential operator
$-\partial_{x}^{2}$ subject to the Dirichlet boundary conditions.

Perturbative computations show that if $u\in{\rm Null}(A_{+}(\epsilon))$,
then $u$ is represented asymptotically as \[
u=\left[\begin{array}{c}
\epsilon(2I-J)v\\
v\end{array}\right]+{\cal O}(\epsilon^{2}),\]
 where $v+2\epsilon(2I-J)v+{\cal O}(\epsilon^{2})=w\in{\rm Null}(M_{+})$.

Now, there exists a solution $\tilde{u}\in\C^{2N}$ of the inhomogeneous
equation $A_{+}(\epsilon)\tilde{u}=u\in{\rm Null}(A_{+}(\epsilon))$
if and only if $u\perp{\rm Null}(A_{+}^{*}(\epsilon))$. For small
$\epsilon\in(0,\epsilon_{0})$, this condition implies that \[
\epsilon(2I-J)v+{\cal O}(\epsilon^{2})=\epsilon(2I-J)w+{\cal O}(\epsilon^{2})\perp w\in{\rm Null}(M_{+}),\]
 which is not possible since $(2I-J)$ is a strictly positive matrix.

\end{document}